\documentclass[11pt]{amsart}

\usepackage[utf8]{inputenc}

\usepackage[english]{babel}

\usepackage{amsmath,amsfonts,amssymb}
\usepackage{mathrsfs}
\usepackage{dsfont} 
\usepackage{mathtools} 
\usepackage{esint} 
\usepackage{xfrac} 
\usepackage{defs} 
\usepackage{fonts} 

\usepackage{amsthm}
\usepackage[shortlabels]{enumitem} 
\usepackage[colorlinks=true,urlcolor=blue, citecolor=red,linkcolor=blue,linktocpage,pdfpagelabels, bookmarksnumbered,bookmarksopen]{hyperref} 
\usepackage[paper=a4paper,hmargin=1in,vmargin=1in, marginparwidth=0.8in]{geometry} 

\usepackage[colorinlistoftodos,prependcaption,linecolor=blue,backgroundcolor=blue!25,bordercolor=blue,textsize=tiny]{todonotes}
\usepackage{cite} 
\usepackage{cleveref}
 
\usepackage{graphicx} 

plus 6pt minus 12pt
\numberwithin{equation}{section}

\newtheorem{theorem}{Theorem}[section]
\newtheorem{proposition}[theorem]{Proposition}
\theoremstyle{definition}
\newtheorem{definition}[theorem]{Definition}
\theoremstyle{plain}
\newtheorem{lemma}[theorem]{Lemma}

\theoremstyle{remark}
\newtheorem{remark}[theorem]{Remark}

\title[An anisotropic Peierls-Nabarro model]{Existence and uniqueness of solutions to the Peierls-Nabarro model in anisotropic media}
\author{Yuan Gao}
\address{Department of Mathematics, Purdue University, West Lafayette, IN}
\email{gao662@purdue.edu}
\author{James M. Scott}
 \address{Department of Applied Physics and Applied Mathematics,
 	Columbia University,
 	New York, NY 10027}
 \email{jms2555@columbia.edu}
\thanks{YG is supported by NSF grant DMS-2204288. JS is supported in part 
	by NSF DMS-1937254 and DMS-2012562.}



\begin{document}

\maketitle

\begin{abstract}
	We study the existence and uniqueness of solutions to the vector field Peierls-Nabarro model for curved dislocations in a transversely isotropic medium. Under suitable assumptions for the misfit potential on the slip plane, we reduce the 3D Peierls-Nabarro model to a nonlocal scalar Ginzburg-Landau equation. For a particular range of elastic coefficients, the nonlocal scalar equation with explicit nonlocal positive kernel is derived. We prove that any stable steady solution has a one-dimensional profile. As a result, we obtain that solutions to the scalar equation, as well as the original 3D system, are characterized as a one-parameter family of straight dislocations.
	This paper generalizes results found previously for the full isotropic case to an anisotropic setting.
\end{abstract}

\section{Introduction}

The vectorial Peierls-Nabarro (PN) model is a nonlinear model that describes the core structure of the dislocation by incorporating the atomistic effect in the dislocation core into the continuum elastic model \cite{peierls1940size, nabarro1947dislocations}. As the most common line defects in materials \cite{HL, Xiang_2006}, the core structures of dislocations are fundamental problems for studying the stability of material structure and the minimum energy barrier for plastic deformations. Particularly for general anisotropic materials, the stationary dislocation core profile and its rigidity is a central problem.
In this paper, we investigate the certain range of the parameters in the anisotropic constitutive relation for which the stationary solution to the vectorial PN model is a one-dimensional profile given by a proper one-dimensional scalar nonlocal Ginzburg-Landau equation with an explicitly computed kernel. 

In the three-dimensional PN model
for the displacement vector $\bu = (u_1,u_2,u_3)$, two half-spaces separated by the slip plane of a dislocation are assumed to be linear elastic continua described by Hooke's law. Here the slip plane is assumed to be a fixed plane $\Gamma$, where the horizontal displacement discontinuity (known as disregistry) happens. 
The two elastic continua are connected by a nonlinear
atomistic potential force across the slip plane, which results a boundary stress nonlinearly depending on the disregistry across the slip plane.  
The direction
and magnitude of the above disregistry due to the dislocation are characterized by the Burgers vector $\bfb$.
The magnitude of the Burgers vector represents the typical length to observe a heavily distorted region in the dislocation core. Hence it is natural to rescale all the quantities including spatial variables $x_1,x_2,x_3$, the displacement vector $\bu = (u_1,u_2,u_3)$ and $\bfb$ with respect to the magnitude of the Burgers vector. After rescaling, we regard all these quantities (with the same notations) as dimensionless and regard the magnitude of the Burgers vector as $4$. This means, after some symmetric assumption on the upper/lower elastic bulk, the disregistry at far field in the shear direction yields the bi-states far field condition $\pm 1$; see \cite{Xiang_2006, gao2021existence, gao2021revisit, dong2021existence} 

In this work the elastic media are transversely isotropic; the stress-strain response of the material is isotropic about an axis normal to the transverse, or basal, plane. The stiffness tensor for transversely isotropic media is uniquely determined by five independent elastic coefficients, see \eqref{eq:StressTensor} below.
We compute the Dirichlet-to-Neumann map associated to the bulk constitutive law in the two half-spaces separated by the slip plane in order to reduce the 3D model to a strongly-coupled 2D nonlocal system. Because of the anisotropy of the media, different slip planes $\Gamma$ result in different forms of reduced models.
In this work we treat two different orientations of the slip plane.
The first case is when $\Gamma$ is oriented perpendicular to the plane of isotropy, and the second case is when $\Gamma$ and the plane of isotropy are parallel; see the illustrations in \Cref{fig:perp} and \Cref{fig:para}. 

\begin{figure}[h]
	\centering
	\begin{minipage}{.45\textwidth}
		\centering
		\includegraphics[width=.8\textwidth]{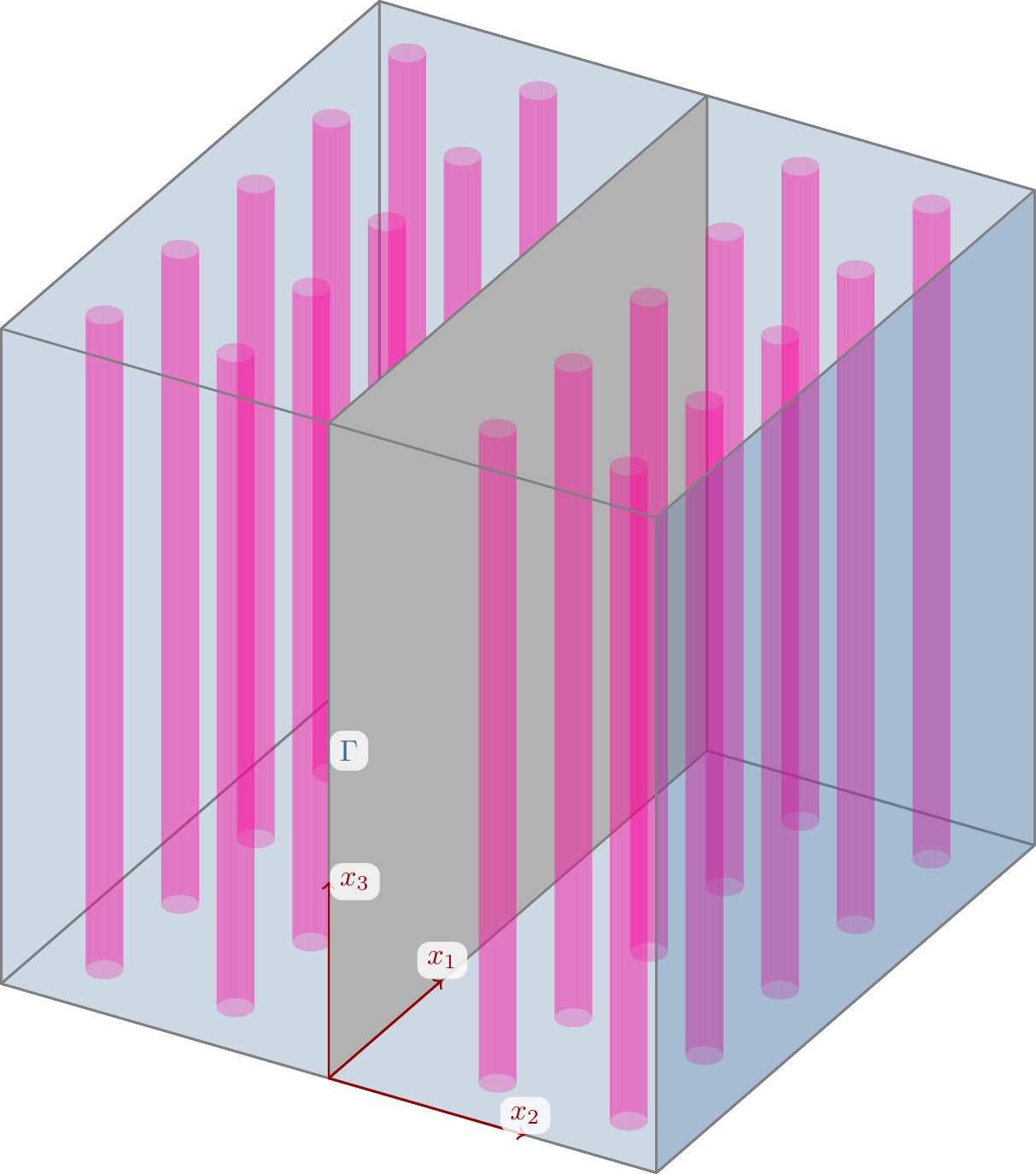}
		\caption{The slipe plane oriented perpendicular to the plane of isotropy.}
		\label{fig:perp}
	\end{minipage}
	\begin{minipage}{.45\textwidth}
		\centering
		\includegraphics[width=.8\textwidth]{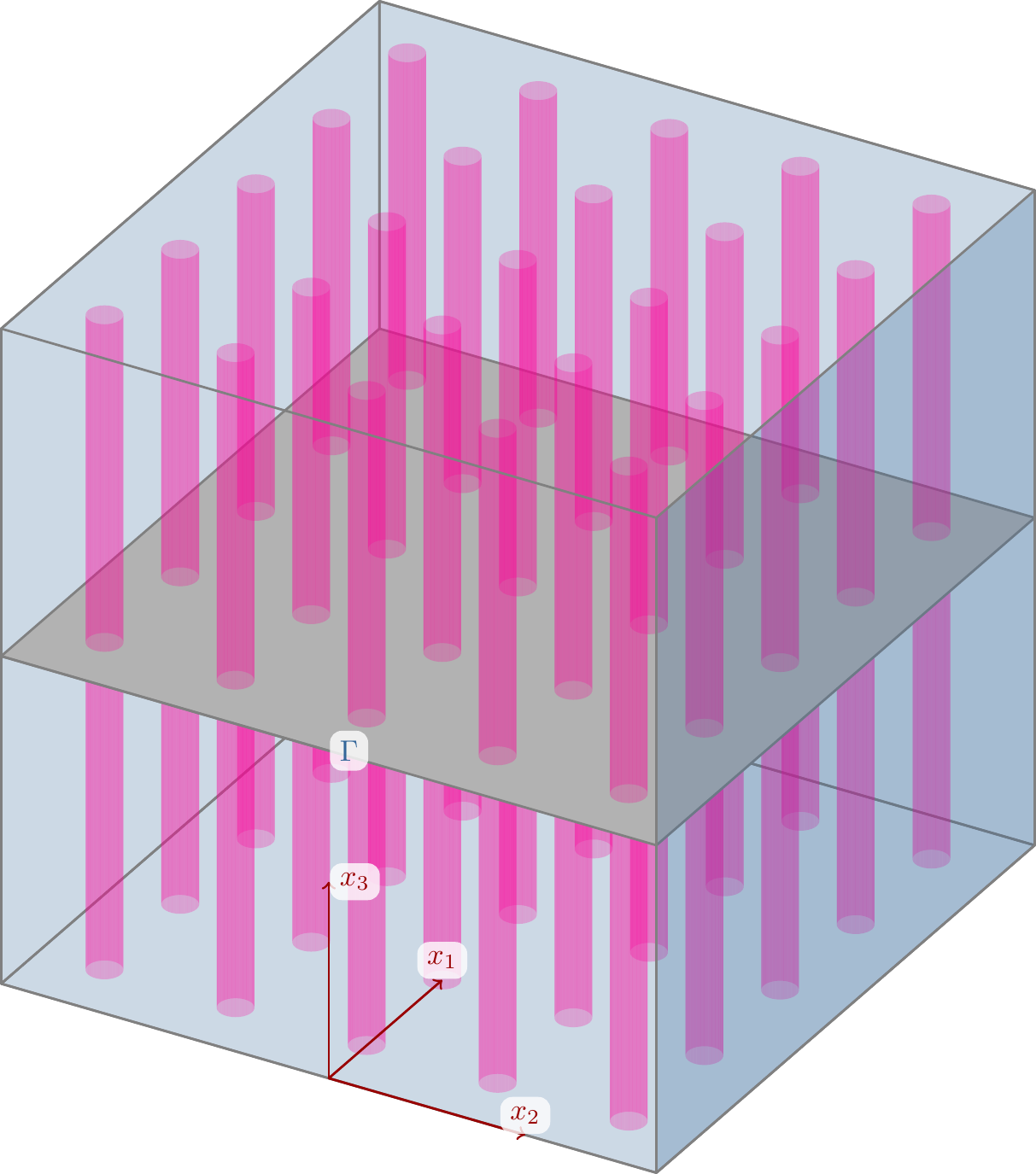}
		\caption{The slipe plane oriented parallel to the plane of isotropy.}
		\label{fig:para}
	\end{minipage}
\end{figure}

The nonlinear misfit potential $W$, defined on the slip plane, creates the boundary stress due to the dislocation disregistry across the slip plane.
In each case, we assume that $W$ depends only on the displacement disregistry in a single direction, and that $W$ is a periodic / double well potential in that one variable.
This assumption permits further reduction of the 2D system to a nonlocal scalar Ginzburg-Landau equation. We compute the exact expression for this integro-differential operator, and find the range of elastic coefficients for which the kernel of the integro-differential operator is positive.
We then apply the procedures that was carried out in \cite{dong2021existence} for the case of fully isotropic elasticity to conclude that the scalar Ginzburg-Landau equation has a unique (up to translation and rotation) solution with a 1D profile. We then follow the elastic extension procedures in \cite{gao2021revisit, gao2021existence} to conclude the well-posedness of the full 3D PN system.

	Again because of the anisotropy of the full 3D media, assuming different directions for the dependence on $W$ leads to different reduced 1D equations.
	In case I, when $\Gamma$ is perpendicular to the plane of isotropy and $W$ depends only on the in-plane-of-isotropy direction of the displacement, the reduced 2D operator is highly-anisotropic, in the sense that its Fourier symbol is characterized by more than three distinct frequency magnitudes.
	In  case II, when $\Gamma$ is perpendicular to the plane of isotropy and $W$ depends only on the out-of-plane-of-isotropy direction of the displacement, the reduced 2D operator is anisotropic but depends on three frequency magnitudes.
	In case III, when $\Gamma$ is parallel to the plane of isotropy, the resulting disregistry depends only on the reference configuration in the plane of isotropy, i.e. the reduced 2D model closely resembles its analogue in the fully isotropic case. This 2D model is itself isotropic, and so choosing $W$ to depend on either direction of the displacement results in equivalent models.
	Our results in these three cases are summarized in three theorems respectively: \Cref{prop:Perp:u1Dependence}, \Cref{prop:Perp:u3Dependence}, and \Cref{prop:parallel}.

	The proof relies on the local BV estimates originally developed in  \cite{CSV19} to study the quantitative flatness of nonlocal minimal surfaces. In \cite[Theorem 4.6]{dong2021existence}, the authors use the flatness estimate for the scalar solution of the nonlocal differential equation by combining the interior BV estimate and a sharp interpolation inequality in \cite{Gui19, FS20, CSV19, savin2018rigidity}. The crucial ingredients are the explicit properties for the kernel representation of the nonlocal operator. To be precise, the nonlocal kernel needs to be homogeneous with respect to dilations, strictly positive, and have good decay estimates;  see \cite{CSV19, dipierro2020improvement, Gui19, gao2021existence}. Taking into account the transverse anisotropy of the materials, for each of the three cases described above, we will take a specific range of the five elastic constants so that we have the desired properties of the kernel ${K}$ that describes the scalar nonlocal operator.
	We use these properties 
	along with \cite[Theorem 4.6]{dong2021existence} to obtain that any bounded stable solution to the reduced scalar nonlocal equation in two dimensions has a 1D profile. 
	For other methods on the rigidity property for nonlocal Allen-Cahn equations, we refer to \cite[Theorem 2]{gao2021existence} for using estimates of the spectrum of a linear operator and refer to \cite{CS05} for using a Liouville-type theorem for the original local equation. With the obtained rigidity result and the explicit kernel representation for the one-dimensional nonlocal Allen-Cahn equation reduced from the 3D vectorial dislocation model, one can further explore the long time behavior of dislocations using maximum principles for the nonlocal operator. We leave this for a future work and refer to the methods in \cite{Mon1, PV1, PV2, gao2022asymptotic} in the case of isotropy.

The PN model for transversely anisotropic media is a generalization of the PN model for full isotropic media. Indeed, each result in this paper has a corresponding result for the PN model in the full isotropic setting by making a particular choice of elastic coefficients; we illustrate this in remarks throughout the paper. In particular, many of the results characterizing the 2D and 1D nonlocal operators studied in \cite{dong2021existence, gao2021revisit, gao2021existence} can be recovered from their analogous statements in this work.
At the same time, in cases I and II we make some special choices of elastic coefficients so that the reduced model can be analyzed relatively easily, 
but the question of whether the results in this work hold for a general choice of coefficients that, for instance, do not satisfy \eqref{eq:Perp:Coeff:Cond2} below, is open.

The paper is organized as follows: in the next section we specify notation and necessary technical lemmas. In \Cref{sec:3DModel:perp} we define the 3D PN model in the case that $\Gamma$ is perpendicular to the plane of isotropy and calculate the reduced 2D model via the Dirichlet-to-Neumann map. \Cref{sec:Perp:u1Eqn} contains the computation and properties of the further reduced 1D nonlocal model in the case that $W$ depends only on $u_1$, and \Cref{sec:Perp:u3Eqn} contains the same in the case that $W$ depends only on $u_3$. In \Cref{sec:parallel} the 3D PN model for the parallel setting is derived, and also contains the dimension reductions, computation and properties of the scalar equation. Finally in \Cref{sec:1DProfile} we conclude that, for each of the three reduced problems, bounded stable solutions have 1D profiles,  and thus complete the proof of \Cref{prop:Perp:u1Dependence}, \Cref{prop:Perp:u3Dependence}, and \Cref{prop:parallel}.

\subsection{Notation and Preliminaries}

For $d \geq 2$, denote the Schwarz class as $\scS(\bbR^d)$ and its dual as $\scS'(\bbR^d)$, and denote the homogeneous Sobolev spaces as
\begin{equation*}
	H^s(\bbR^d) := \left\{ u \in \scS'(\bbR^d) \, : \, \int_{\bbR^d} |\xi|^{2s} |\hat{u}(\xi)|^2 \,\rmd \xi < \infty \right\}\,. 
\end{equation*}
The Fourier transform for $L^1(\bbR^2)$ functions $v$ is written as
\begin{equation*}
	\hat{v}(k) = \cF v(k) = \intdm{\bbR^2}{\rme^{-\imath x \cdot k} v(x) }{x}\,,
\end{equation*}
with inverse denoted by
\begin{equation*}
	v^{\vee}(k) = \cF^{-1} v(k) = (2 \pi)^{-2} \intdm{\bbR^2}{\rme^{\imath x \cdot k} v(x) }{x}\,.
\end{equation*}

\begin{lemma}\label{thm:AnisoLaplaceRep}
	Let $\rho > 0$.  The integral representation of the square root of the operator $(-\Delta)_{\rho} := ( -\p_{x_1}^2 - \rho \p_{x_2}^2 )$ acting on sufficiently smooth $u : \bbR^2 \to \bbR$ is
	\begin{equation*}
		(-\Delta_{\rho})^{1/2} u(x) := -\frac{1}{4 \pi \sqrt{\rho}} \intdm{\bbR^2}{ \frac{u(x-y)+u(x+y)-2u(x)}{ ( y_1^2 + \frac{y_3^2}{\rho} )^{3/2} } }{y}\,,
	\end{equation*}
	and the associated Fourier symbol is $\sqrt{k_1^2 + \rho k_2^2}$.
\end{lemma}

\section{Slip plane perpendicular to the plane of isotropy: the 3D model}\label{sec:3DModel:perp}

For this section we introduce the 3D Peierls-Nabarro model for transversely anisotropic elasticity, and then in \Cref{subsec:perp:2DRed} we reduce it to a 2D nonlocal model via the Dirichlet-to-Neumann map.  The first two main results (\Cref{prop:Perp:u1Dependence} and \Cref{prop:Perp:u3Dependence}) for the case of the slip plane perpendicular to the plane of isotropy are summarized in this section. The proofs of these two results  will be completed in Section \ref{sec:1DProfile}.

In the PN model, the two half spaces of the elastic medium are separated by the slip plane $\Gamma$.
Let $\bu = (u_1,u_2,u_3)$ be the displacement vector. The total energy of the system is
\begin{equation}\label{eq:3D:TotalEnergy}
	E(\bu) := E_{\mathrm{els}}(\bu) + E_{\mathrm{mis}}(\bu)\,.
\end{equation}
The material properties of transversely isotropic materials are determined by five independent elastic constants $C_{11}$, $C_{13}$, $C_{33}$, $C_{44}$, and $C_{66}$. These constants determine the stress-strain relation comprising the total elastic energy  $E_{\mathrm{els}}(\bu)$ for the two half-spaces. To be precise, $E_{\mathrm{els}}(\bu)$ is defined as
\begin{equation}\label{eq:3DEnergy}
	E_{\mathrm{els}}(\bu) = \frac{1}{2} \int_{\bbR^3 \setminus \Gamma} \sigma  :  \veps \, \rmd x = \frac{1}{2} \int_{\bbR^3 \setminus \Gamma} \sigma_{ij}  \veps_{ij} \, \rmd x\,,
\end{equation}
where $\veps_{ij}$ is the strain tensor
\begin{equation*}
	\veps_{ij} = \frac{1}{2}(\p_j u_i + \p_i u_j)\,, \quad \text{ for } i,j, = 1,2,3, \quad \p_i := \frac{\p}{\p x_i}\,,
\end{equation*}
and $\sigma$ is the stress tensor in $\bbR^{3 \times 3}$ given by
\begin{equation}\label{eq:StressTensor}
	\begin{split}
		\sigma_{11} &= C_{11} \veps_{11} + (C_{11} - 2 C_{66}) \veps_{22} + C_{13} \veps_{33}\,,\\
		\sigma_{22} &= (C_{11}-2 C_{66}) \veps_{11} + C_{11} \veps_{22} + C_{13} \veps_{33}\,,\\
		\sigma_{33} &= C_{13} \veps_{11} + C_{13} \veps_{22} + C_{33} \veps_{33} \,,\\
		\sigma_{23} &= \sigma_{32} = C_{44} \veps_{23} \,,\\
		\sigma_{13} &= \sigma_{31} = C_{44} \veps_{13} \,,\\
		\sigma_{12} &= \sigma_{21} = C_{66} \veps_{12} \,.
	\end{split}
\end{equation}
We are using Einstein summation notation in \eqref{eq:3DEnergy}:
\begin{equation*}
	\sigma_{ij} \veps_{ij} = \sum_{i,j=1}^3 \sigma_{ij} \veps_{ij}\,.
\end{equation*}
The five elastic moduli defining $\sigma$ are assumed to satisfy the uniform ellipticity conditions (see \cite{payton2012elastic, merodio2003note})
\begin{equation}\label{eq:UniformEllipticityConditions}
	0< C_{66} < C_{11}\,, \qquad C_{13}^2 < C_{33} (C_{11}-C_{66})\,, \qquad C_{44} > 0\,.
\end{equation}
Note that one can recover the PN model for fully isotropic media by setting
\begin{equation}\label{eq:ConsistentWithIsotropy}
	C_{11} = C_{33} = 2 \mu \frac{1-\nu}{1-2 \nu}\,, \quad C_{13} = 2 \mu \frac{\nu}{1-2 \nu}\,, \quad C_{44} = C_{66} = \mu\,,
\end{equation}
where $\mu$ is the shear modulus and $\nu$ is the Poisson ratio.

In this first setting, $\Gamma$ is taken perpendicular to the plane of isotropy, that is, $\Gamma = \{ (x_1,x_2,x_3) \, : \, x_2 = 0 \}$; see Fig. \ref{fig:perp}.
The misfit energy $E_{\mathrm{mis}}(\bu)$ across the slip plane due to atomistic reactions is therefore given by
\begin{equation*}
	E_{\mathrm{mis}}(\bu) := \int_{\Gamma} \gamma(u_1^+ - u_1^-,u_3^+-u_3^-) \rmd S =  \int_{\Gamma} W(u_1^+,u_3^+) \rmd S\,,
\end{equation*}
where $u_i^{\pm}(x_1,x_3) := u_i(x_1,0^{\pm},x_3)$ for $i = 1$, $3$. We assume that the nonlinear potential $W$ is twice differentiable with H\"older-continuous derivatives, i.e. $W \in C^{2,a}_{b}(\bbR^2;\bbR)$ for some $a \in (0,1)$.

The equilibrium structure of a dislocation is obtained by minimizing \eqref{eq:3DEnergy} subject to the following conditions at the slip plane
\begin{equation}\label{eq:BCs}
	\begin{split}
	u_1^+(x_1,x_3) &= - u_1^-(x_1,x_3)\,, \\
	u_2^+(x_1,x_3) &= u_2^-(x_1,x_3)\,, \\
	u_3^+(x_1,x_3) &= - u_3^-(x_1,x_3)\,. \\
	\end{split}
\end{equation}

We will also consider two different scenarios for the form of $W$; we will treat the cases $\p_{u_1} W = 0$ and $\p_{u_3}W=0$ separately in \Cref{sec:Perp:u1Eqn} and \Cref{sec:Perp:u3Eqn}. 
In the former case we assume the bi-states far field boundary condition for $u_1$
\begin{equation}\label{eq:FarFieldCond:u1}
	u_1^+(\pm \infty,x_3) = \pm 1 \quad \text{ for any } x_3 \in \bbR\,,
\end{equation}
and in the latter case we assume a similar condition for $u_3$
\begin{equation}\label{eq:FarFieldCond:u3}
	u_3^+(x_1,\pm \infty) = \pm 1 \quad \text{ for any } x_1 \in \bbR\,.
\end{equation}

Formally, the elastic energy is infinite due to the slow decay rate of $\veps$. We therefore define a minimizer via perturbation. To be precise, for any $\bsvarphi \in C^{\infty}(\bbR^3 \setminus \Gamma ; \bbR^3)$ with compact support in a Euclidean ball $B_R(0) \subset \bbR^3$ we define the perturbed elastic energy
\begin{equation*}
	\begin{split}
	\hat{E}_{\mathrm{els}}(\bsvarphi|\bu) &:= \int_{\bbR^3 \setminus \Gamma} \frac{1}{2} (\sigma_{\bu} + \sigma_{\bsvarphi}) : (\veps_{\bu} + \veps_{\bsvarphi}) - \frac{1}{2} \sigma_{\bu} : \veps_{\bu}\, \rmd x \\
	&= \int_{\bbR^3 \setminus \Gamma} \frac{1}{2} \big( (\sigma_{\bsvarphi})_{ij} (\veps_{\bsvarphi})_{ij} 
	+ (\sigma_{\bu})_{ij} (\veps_{\bsvarphi})_{ij}
	+(\sigma_{\bsvarphi})_{ij} (\veps_{\bu})_{ij} \big) \, \rmd x \\
	&:= E_{\mathrm{els}}(\bsvarphi) + \cC_{\mathrm{els}}(\bu,\bsvarphi)\,,
	\end{split}
\end{equation*}
where the cross term is defined as
\begin{equation*}
	\cC_{\mathrm{els}}(\bu,\bsvarphi) := 
	 \int_{\bbR^3 \setminus \Gamma} \frac{1}{2} \big( 
	\sigma_{\bu} : \veps_{\bsvarphi}
	+ \sigma_{\bsvarphi} :  \veps_{\bu} \big) \, \rmd x
	=\int_{\bbR^3 \setminus \Gamma} \frac{1}{2} \big( 
	(\sigma_{\bu})_{ij} (\veps_{\bsvarphi})_{ij}
	+(\sigma_{\bsvarphi})_{ij} (\veps_{\bu})_{ij} \big) \, \rmd x
\end{equation*}
and $\sigma_{\bu}$, $\veps_{\bu}$ and $\sigma_{\bsvarphi}$, $\veps_{\bsvarphi}$ are the stress and strain tensors corresponding to $\bu$ and $\bsvarphi$, respectively.
We also define the perturbed misfit energy as
\begin{equation*}
	\hat{E}_{\mathrm{mis}}(\bsvarphi|\bu) := \int_{\Gamma} W(u_1^+ + \varphi_1^+, u_3^+ + \varphi_3^+) - W(u_1^+,u_3^+) \, \rmd S.
\end{equation*}
The perturbed total energy is then defined as
\begin{equation*}
	\hat{E}(\bsvarphi| \bu) := \hat{E}_{\mathrm{els}}(\bsvarphi|\bu) + \hat{E}_{\mathrm{mis}}(\bsvarphi|\bu)\,.
\end{equation*}

\begin{remark}
	Since $\bu$ and $\bu + \bsvarphi$ coincide outside of $B_R$, the perturbed elastic energy $\hat{E}(\bsvarphi| \bu)$ is equivalent to the local perturbed elastic energy
	\begin{equation*}
		\hat{E}_{\mathrm{els}}(\bsvarphi|\bu;B_R) := \int_{B_R \setminus \Gamma} \frac{1}{2} (\sigma_{\bu} + \sigma_{\bsvarphi}) : (\veps_{\bu} + \veps_{\bsvarphi}) - \frac{1}{2} \sigma_{\bu} : \veps_{\bu}\, \rmd x := E_{\mathrm{els}}(\bu+\bsvarphi;B_R) - E_{\mathrm{els}}(\bu;B_R)\,,
	\end{equation*}
	and similarly for the misfit energy
	\begin{equation*}
		\hat{E}_{\mathrm{mis}}(\bsvarphi|\bu;B_R) := \int_{B_R \cap \Gamma} W(u_1^+ + \varphi_1^+, u_3^+ + \varphi_3^+) - W(u_1^+,u_3^+) \, \rmd S := E_{\mathrm{mis}}(\bu+\bsvarphi;B_R) - E_{\mathrm{mis}}(\bu;B_R)\,,
	\end{equation*}
	so that $\hat{E}(\bsvarphi|\bu)$ is equivalent to
	\begin{equation*}
		\hat{E}(\bsvarphi|\bu;B_R) := \hat{E}_{\mathrm{els}}(\bsvarphi|\bu;B_R) +\hat{E}_{\mathrm{mis}}(\bsvarphi|\bu;B_R) :=  E(\bu+\bsvarphi;B_R) - E(\bu;B_R)\,.
	\end{equation*}
\end{remark}
Hence, we say that $\bu$ is a local minimizer, following the convention in \cite{gao2021revisit}.

\begin{definition}\label{def:perp:LocalMin}
	We call $\bu$ a \textit{local minimizer of} $E$ if $\bu$ satisfies
	\begin{equation*}
		E(\bu+\bsvarphi; B_R) - E(\bu;B_R) \geq 0
	\end{equation*}
	for any $\bsvarphi \in C^{\infty}(\bbR^3 \setminus \Gamma ; \bbR^3)$ with compact support in a Euclidean ball $B_R(0) \subset \bbR^3$ satisfying the conditions
	\begin{equation}\label{eq:BCs:perturbed}
		\begin{split}
			\varphi_1^+(x_1,x_3) &= - \varphi_1^-(x_1,x_3)\,, \\
			\varphi_2^+(x_1,x_3) &= \varphi_2^-(x_1,x_3)\,, \\
			\varphi_3^+(x_1,x_3) &= - \varphi_3^-(x_1,x_3)\,. \\
		\end{split}
	\end{equation}
\end{definition}

In order to define the Euler-Lagrange equation associated to the total energy $E$ we need to set some notation. 
We define $\bbL\bu$ to be the second-order partial differential system
\begin{equation*}
	\bbL \bu := \div \sigma\,, \qquad (\bbL \bu)_i = L_{ij} u_j, \quad i = 1,2,3,
\end{equation*}
where $L_{ij}$ with $i$,$j = 1$, $2$, $3$, is the matrix of partial derivatives
\begin{equation*}
	\begin{split}
	L_{11} &= C_{11} \p_{11} + C_{66} \p_{22} + C_{44} \p_{33}\,, \\
	L_{22} &= C_{66} \p_{11} + C_{11} \p_{22} + C_{44} \p_{33}\,, \\
	L_{33} &= C_{44} \p_{11} + C_{44} \p_{22} + C_{33} \p_{33}\,, \\
	L_{23} &= L_{32} = (C_{13}+C_{44}) \p_{23}\,, \\
	L_{13} &= L_{31} = (C_{13}+C_{44}) \p_{13}\,, \\
	L_{12} &= L_{21} = (C_{11}-C_{66}) \p_{12}\,.
	\end{split}
\end{equation*}

\begin{lemma}
	Assume that $\bu \in C^2(\bbR^3 \setminus \Gamma ; \bbR^3)$ satisfying \eqref{eq:BCs} and either \eqref{eq:FarFieldCond:u1} or \eqref{eq:FarFieldCond:u3} is a local minimizer of the total energy $E$ in the sense of \Cref{def:perp:LocalMin}. Then $\bu$ satisfies the Euler-Lagrange equations
	\begin{equation}\label{eq:ELEqn}
	\begin{gathered}
		\bbL \bu = 0 \text{ in } \bbR^3 \setminus \Gamma\,, \\
		\sigma_{12}^+ + \sigma_{12}^- = \p_{u_1} W(u_1^+,u_3^+) \text{ on } \Gamma\,, \\
		\sigma_{32}^+ + \sigma_{32}^- = \p_{u_3} W(u_1^+,u_3^+) \text{ on } \Gamma\,, \\
		\sigma_{22}^+ - \sigma_{22}^- = 0 \text{ on } \Gamma\,.
	\end{gathered}
	\end{equation}
\end{lemma}

\begin{proof}
	The computation is similar to its analogue in the case of fully isotropic elasticity; see \cite[Appendix A]{dong2021existence} for the isotropic case.
\end{proof}

The analysis of the reduced models changes significantly depending on the choice of the elastic coefficients. We leave more general cases to future works, but here we make the choices
\begin{equation}\label{eq:Perp:Coeff:Cond2}
	\sqrt{C_{11} C_{33} } - C_{13} - 2 C_{44} = 0\,, \qquad \text{ and } \qquad C_{11} = C_{33}\,.
\end{equation}
The first relation greatly simplifies the analysis, while the second relation further simplifies the presentation of the results for the case of $\Gamma$ perpendicular to the plane of isotropy.

Our first two main results are the following:
\begin{proposition}[Case I: $\Gamma$ perpendicular to the plane of isotropy, $\p_{u_3}W = 0$]\label{prop:Perp:u1Dependence}
	Suppose that the five elastic constants satisfy \eqref{eq:UniformEllipticityConditions} and \eqref{eq:Perp:Coeff:Cond2}, and assume that the quantities
	\begin{equation}\label{eq:Perp:Coeff:NuAndDelta}
		\nu := \frac{C_{13}}{2( C_{13} + C_{44} )}\,, \qquad \delta := \frac{C_{66}}{C_{44}}\,, \qquad \mu := C_{44}
	\end{equation}
	satisfy
	\begin{equation}\label{eq:Perp:u1Dependence:KernelPosRegion}
		\delta > 0 \quad \text{ and } \quad  \max \left\{ 1 - \frac{2}{\delta}\,, \frac{\sqrt{\delta}}{2} \frac{2\delta -3 }{ 2 \delta^{3/2} - 2 \delta^{1/2} + 1 } \right\} < \nu < \frac{2}{4 - \delta + \sqrt{\delta^2 + 8}}\,.
	\end{equation}
	Assume that $\p_{u_3}W = 0$, i.e. $W(u_1,u_3) = W(u_1)$, and that
	\begin{equation}\label{potential}
		\begin{gathered}
		W \text{ is a periodic / double-well potential in } C^{2,a}_b(\bbR) \text{ satisfying } \\
		W(v)>W(\pm 1) \text{ for } x\in(-1,1), \quad W''(\pm 1)>0.
		\end{gathered}
	\end{equation}
	Then  the system \eqref{eq:ELEqn} has a unique (up to translation in the $x_1$-direction and rotation about the $x_2$-axis) 
	classical solution $\bu : \bbR^3 \to \bbR^3$ belonging to $C^{\infty}(\bbR^3 \setminus \Gamma)$ with $\bu(\cdot,x_2,\cdot) \in \dot{H}^s(\bbR^2)$ for all $s \geq 1$ and for all fixed $x_2 \neq 0$. Moreover, the solution is a local minimizer of $E$ in the sense of \Cref{def:perp:LocalMin}.
	The first component $u_1$ is the unique stable solution to \eqref{eq:nonlocalMain1d} with specific nonlocal kernel defined in \eqref{eq:perp:u1Eqn:Kernel} and $u_1$ has a 1D profile.
\end{proposition}

\begin{proposition}[Case II: $\Gamma$ perpendicular to the plane of isotropy, $\p_{u_1}W = 0$]\label{prop:Perp:u3Dependence}
	Suppose that the five elastic constants satisfy \eqref{eq:UniformEllipticityConditions} and \eqref{eq:Perp:Coeff:Cond2}. Define the quantities $\mu$, $\nu$ and $\delta$ as in \eqref{eq:Perp:Coeff:NuAndDelta}, and define
	\begin{equation}\label{eq:Defns:pAndQ}
		p := \delta (2(1-\nu)-\delta(1-2\nu)) \quad \text{ and } \quad q := 1 - \nu\,.
	\end{equation}
	Suppose that
	\begin{equation}\label{eq:Perp:u3Dependence:KernelPosRegion}
		(\nu,\delta) \in \cU^+\,,
	\end{equation}
	where $\cU^+ \subset \bbR^2$ is the open region enclosed by the curves $\nu = \frac{1}{2}$,  $\nu = \frac{\delta(\delta-2)}{2(\delta^2-\delta+1)}$, $\nu = 1 - \frac{2}{\delta}$, $\nu = \frac{2}{ 1 + 4 \delta - 2 \delta^2 + \sqrt{1 - 8 \delta + 28 \delta^2 - 16 \delta^3 + 4 \delta^4} }$, and
	$p = \wt{r}$, where $\wt{r}(q)$ is the smallest positive root of the polynomial
	\begin{multline*}
		x \mapsto -8  (-1 + q) x^4 + 8 (-1 + q) q^3 \\
		+  q^2 (13 + 14 q - 11 q^2) x + (-11 + 14 q + 13 q^2) x^3 + 
		2 q (1 - 18 q + q^2) x^2 \,.
	\end{multline*}
	Assume that $\p_{u_1}W = 0$, i.e. $W(u_1,u_3) = W(u_3)$ and that $W$ satisfies \eqref{potential}. Then the system \eqref{eq:ELEqn} has a unique (up to translation in the $x_3$-direction and rotation about the $x_2$-axis)  classical solution $\bu : \bbR^3 \to \bbR^3$ belonging to $C^{\infty}(\bbR^3 \setminus \Gamma)$ with $\bu(\cdot,x_2,\cdot) \in \dot{H}^s(\bbR^2)$ for all $s \geq 1$ and for all fixed $x_2 \neq 0$. Moreover, the solution is a local minimizer of $E$ in the sense of \Cref{def:perp:LocalMin}.  The third component $u_3$ is the unique stable solution to \eqref{eq:nonlocalMain1d:u3Eqn} with specific nonlocal kernel defined in \eqref{eq:AlmostIso:KernelDefn1} and $u_3$ has a 1D profile.
\end{proposition}

Note that \eqref{eq:Perp:Coeff:Cond2}-\eqref{eq:Perp:Coeff:NuAndDelta} imply that 
\begin{equation*}
	C_{11} = C_{33} = 2 \mu \frac{1-\nu}{1-2 \nu}\,, \quad C_{13} = 2 \mu \frac{\nu}{1-2 \nu}\,, \quad C_{44} =\mu\,, \quad  C_{66} = \delta \mu\,.
\end{equation*}
The case $\delta=1$ corresponds to the case of isotropic elasticity, and $\mu$ and $\nu$ respectively correspond to the shear modulus and Poisson's ratio of an isotropic solid.

The conditions in  \eqref{eq:UniformEllipticityConditions} for uniform ellipticity  are equivalent to the conditions
\begin{equation}\label{eq:Perp:UniformEllipticity:NuAndDelta}
	0 < \mu\,, \quad 0 < \delta < 4\,, \quad 1 - \frac{2}{\delta} < \nu < \frac{1}{2}\,. 
\end{equation}
Note that $\delta$ can be arbitrarily close to $0$, so there is no lower bound on $\nu$; this is expected behavior of Poisson's ratios in anisotropic media \cite{ting2005poisson}.
It is straightforward to check that the regions in the $(\nu,\delta)$-plane described by \eqref{eq:Perp:u1Dependence:KernelPosRegion} and $\cU^+$ are subsets of the region defined by \eqref{eq:Perp:UniformEllipticity:NuAndDelta}.

\subsection{Reduction of the 3D System to a Nonlocal 2D System}\label{subsec:perp:2DRed}

We can reduce to a nonlocal 2D system that depends only on $u_1^+$ and $u_3^+$ by using the relations in \eqref{eq:ELEqn} satisfied by the Dirichlet-to-Neumann maps $(-\sigma_{12}^+,-\sigma_{22}^+,-\sigma_{32}^+)$ and $(\sigma_{12}^-,\sigma_{22}^-,\sigma_{32}^-)$ for the upper and lower half-spaces respectively. 
We calculate this map using the Fourier transform.
Here, $x = (x_1,x_3)$ and $k = (k_1,k_3)$, so that statements for the slip plane inherit the conventions used for the full 3D system.
We state the dimension reduction in the following lemma.
\begin{lemma}\label{lma:FindingGeneral2DSystem}
	Let $C_{11}$, $C_{13}$, $C_{33}$, $C_{44}$ and $C_{66}$ satisfy \eqref{eq:UniformEllipticityConditions} and \eqref{eq:Perp:Coeff:Cond2}, and define $\mu$, $\nu$ and $\delta$ as in \eqref{eq:Perp:Coeff:NuAndDelta}. 
	Suppose that $\bu$ satisfies \eqref{eq:ELEqn} and remains bounded as $|x_2| \to \infty$, and assume that $u_1^{+}$ and $u_3^{+}$ belong to the fractional Sobolev space $H^s(\bbR^2)$ for some $s \geq 1/2$. Then $\bu$ can be expressed entirely in terms of $u_1^+$ and $u_3^+$. In particular, $(u_1^+, u_3^+)$ satisfies the nonlocal system
	\begin{equation}\label{eq:2DEqn:Formal}
		\begin{bmatrix}
			\cF (\sigma_{12}^+(k)+\sigma_{12}^-(k))  \\
			\cF (\sigma_{32}^+(k)+\sigma_{32}^-(k))  \\
		\end{bmatrix}
		:= - \bbA(k) 
		\begin{bmatrix}
			\hat{u}_1^+(k) \\
			\hat{u}_3^+(k) \\
		\end{bmatrix}
		= \cF 
		\begin{bmatrix}
			\p_{u_1} W (u_1^+,u_3^+) \\
			\p_{u_3} W (u_1^+,u_3^+) \\
		\end{bmatrix} \quad \text{ on } \Gamma\,,
	\end{equation}
	where the $2 \times 2$ matrix $\bbA(k)$
	is given by
	\begin{equation}\label{eq:Perp:2DModel:OperatorMatrix}
		\bbA(k) := \frac{2 \mu}{1-\nu} r_1 \cdot
		\begin{bmatrix}
			\frac{ \delta (2(1-\nu)-\delta(1-2\nu)) k_1^2 + (1-\nu) k_3^2}{r_1^2} & \frac{ \delta \nu  k_1 k_3}{ r_1^2} \\
			\frac{ k_1 k_3 }{\delta r_1^2} \cdot \frac{(\nu \delta +  (1-\nu)(1-\delta) ) r_2 + (\nu \delta^2 - (1-\nu)(1-\delta)^2) r_1 }{r_1+r_2}  & \frac{r_1 r_2 - \nu k_1^2}{r_1^2} \\
		\end{bmatrix}\,,
	\end{equation}
	and   
	\begin{equation}\label{eq:perp:RootDefn}
		r_1 := r_1(k) := \sqrt{k_1^2+k_3^2/\delta} \quad \text{ and } \quad r_2 := r_2(k) := \sqrt{k_1^2+ k_3^2}\,.
	\end{equation}
\end{lemma}

Note that when $\delta = 1$ the matrix $\bbA$ is consistent with the Fourier transform of the Dirichlet-to-Neumann map in the setting of dislocations in fully isotropic media, see \cite{dong2021existence}.

\begin{proof}
Dividing the system by $C_{44}$, taking the Fourier transform of $\bbL \bu = 0$ in the $(x_1,x_3)$ variables, and using \eqref{eq:Perp:Coeff:Cond2} and \eqref{eq:Perp:Coeff:NuAndDelta} gives the set of three equations
\begin{equation*}
	\begin{split}
		-\left( \frac{2(1-\nu)}{1-2 \nu} k_1^2 +k_3^2 \right) \hat{u}_1 + \delta \p_{22} \hat{u}_1 + \left( \frac{2(1-\nu)}{1-2\nu} - \delta \right) (\imath k_1) \p_2 \hat{u}_2 +  \frac{1}{1-2 \nu}  (k_1 k_3) \hat{u}_3 = 0\,, \\
		\left( \frac{2(1-\nu)}{1-2\nu} - \delta \right) (i k_1) \p_2 \hat{u}_1 - \left( \delta k_1^2 + k_3^2 \right) \hat{u}_2 + \frac{2(1-\nu)}{1-2 \nu} \p_{22} \hat{u}_2 +  \frac{1}{1-2 \nu}   (\imath k_3) \p_2 \hat{u}_3 = 0\,, \\
		- \frac{1}{1-2 \nu}   (k_1 k_3) \hat{u}_1 +  \frac{1}{1-2 \nu}   ( \imath k_3) \p_2 \hat{u}_2 - \left( k_1^2 + \frac{2(1-\nu)}{1-2 \nu} k_3^2 \right) \hat{u}_3 + \p_{22} \hat{u}_3 = 0\,.
	\end{split}
\end{equation*}
This is a $3 \times 3$ system of second-order ODEs in $x_2$. We solve this system in the upper and lower half-spaces separately. 

To solve this system in the upper half-space, we write it as a $6 \times 6$ first-order system
\begin{equation}\label{eq:6by6System}
	\p_2
	\begin{bmatrix}
		\hat{\bu}(k,x_2) \\
		\p_2 \hat{\bu}(k,x_2) 
	\end{bmatrix}
	= \frak{A}(k) 
	\begin{bmatrix}
		\hat{\bu}(k,x_2) \\
		\p_2 \hat{\bu}(k,x_2)
	\end{bmatrix}
\end{equation}
where the $6 \times 6$ matrix $\frak{A}$ depends on $k$, $\mu$, $\nu$ and $\delta$, with boundary conditions
\begin{equation*}
	\hat{\bu}(k_1,k_3,0) = \hat{\bu}^+(k_1,k_3)\,, \qquad \hat{\bu}(k,+\infty) = 0\,.
\end{equation*}
Thanks to the conditions \eqref{eq:UniformEllipticityConditions} and \eqref{eq:Perp:Coeff:Cond2}, the characteristic polynomial of $\frak{A}(k)$ has six real nonzero roots $\pm r_1(k)$ and $\pm r_2(k)$ as in \eqref{eq:perp:RootDefn}, 
and $\pm r_2$ both have multiplicity $2$.
So the general solution to \eqref{eq:6by6System} is an arbitrary linear combination of six terms, each of which is a product of a (generalized) eigenvector and $\rme^{\pm r_i(k) x_2}$, where $i = 1,2$. The computation is straightforward.

Since $\bu$ remains bounded as $|x_2| \to \infty$, we exclude the solutions with positive roots from the final solution, i.e. three of the coefficients in the linear combination are zero. 
The remaining three coefficients are completely determined by the boundary conditions at $x_2 = 0$, and so the unique solution to \eqref{eq:6by6System} has the form $\hat{\bu}(k,x_2) = \bbB^+(k,x_2) \hat{\bu}^+(k)$ for $x_2 > 0$, for a $3 \times 3$ matrix $\bbB^+$. The form of the matrix $\bbB^+$ is determined only from the properties of the eigenvalues and eigenvectors of $\frak{A}$.

In exactly a similar way, the unique solution to the system \eqref{eq:6by6System} for the lower half-space with boundary conditions
\begin{equation*}
	\hat{\bu}(k_1,k_3,0) = \hat{\bu}^-(k_1,k_3)\,, \qquad \hat{\bu}(k,-\infty) = 0\,,
\end{equation*}
has the form $\hat{\bu}(k,x_2) = \bbB^-(k,x_2) \hat{\bu}^-$ for $x_2 < 0$, for some $3 \times 3$ matrix $\bbB^-$. In fact, $\bbB^-(k,x_2) = \overline{\bbB^+(k,-x_2)}$, where $\overline{z}$ is the complex conjugate of $z$.

Next, we use the slip plane boundary conditions \eqref{eq:BCs} to see that $\hat{\bu}(k,t)$ depends only on $\hat{u}_1^+$, $\hat{u}_2^+$, and $\hat{u}_3^+$ for all $k \in \bbR^2$ and all $x_2 \in \bbR$.
Finally, the equation $\cF (\sigma_{22}^+ - \sigma_{22}^-) = 0$ in \eqref{eq:ELEqn} allows us to write $\hat{u}_2^+$ in terms of $\hat{u}_1^+$ and $\hat{u}_3^+$, and inserting this relation into the other two equations involving the Dirichlet-to-Neumann map in \eqref{eq:ELEqn} gives the expression \eqref{eq:Perp:2DModel:OperatorMatrix} for the matrix $\bbA(k)$.
\end{proof}

The matrix $\bbA$ is positive-definite for $k \neq 0$ thanks to the uniform ellipticity conditions \eqref{eq:UniformEllipticityConditions}. Thus by Plancherel's identity
\begin{equation*}
	C^{-1} \Vnorm{(u_1,u_3)}_{\dot{H}^{\frac{1}{2}}(\Gamma)}^2 \leq \int_{\Gamma} \Vint{ \bbA(k) (\hat{u}_1,\hat{u}_3), (\hat{u}_1,\hat{u}_3) } \, \rmd k_1 \, \rmd k_3 \leq C  \Vnorm{(u_1,u_3)}_{\dot{H}^{\frac{1}{2}}(\Gamma)}^2\,.
\end{equation*}
In the next two sections we treat Cases I and II as described in the introduction. In case I, $W$ depends only on $u_1^+$, and in case II, $W$ depends only on $u_3^+$. 
In both cases the simplification allows us to further reduce the 2D equation to a one-dimensional problem.

\section{The 1D reduced equation in the perpendicular case: $\p_{u_3} W = 0$}\label{sec:Perp:u1Eqn}

Under the assumption that $W(u_1,u_3) = W(u_1)$, we derive the reduced scalar equation for $u_1^+$, which is an anisotropic nonlocal Ginzburg-Landau equation. We will derive the nonlocal kernel representation for this anisotropic nonlocal equation in Proposition \ref{prop:Aniso:1DOperator:IntegralForm}. Then we study the explicit formula for the nonlocal kernel and prove the positivity of this kernel in \Cref{prop:KernelForm:perp:u1Dep} and \Cref{prop:temp:u1}.

At this point, we drop the superscripts on $u_1^+$, $u_3^+$ for simplicity. 
However, we retain the index convention for the spatial and frequency variables $x = (x_1,x_3) \in \bbR^2$ and $k = (k_1,k_3) \in \bbR^2$.
Our reduced nonlocal system reads
\begin{equation}\label{eq:2DEqn:Case1}
	\bbA(k) 
	\begin{bmatrix}
		\hat{u}_1(k) \\
		\hat{u}_3(k) \\
	\end{bmatrix}
	= - \cF 
	\begin{bmatrix}
		\p_{u_1} W (u_1,u_3) \\
		\p_{u_3} W (u_1,u_3) \\
	\end{bmatrix} \quad \text{ on } \Gamma\,.
\end{equation}
Since $\p_{u_3} W = 0$, we can solve for $\hat{u}_3$ in terms of $\hat{u}_1$ and substitute into the remaining equation. We obtain the one-dimensional scalar anisotropic nonlocal equation for $\hat{u}_1$ in Fourier space
\begin{equation}\label{eq:OneDimEqnForU1}
	\cF(\cL u_1)(k_1,k_3) := \frac{\det \bbA(k)}{a_{22}(k)} \hat{u}_1(k) = - \cF (W'(u_1))\,.
\end{equation}
Therefore, we invert the Fourier transform and obtain a new 1D nonlocal equation
\begin{equation}\label{eq:nonlocalMain1d}
	\begin{gathered}
		\cL u_1 (x_1,x_3) = - W'(u_1(x_1,x_3))\,, \quad (x_1,x_3) \in \Gamma\,, \\
		\lim\limits_{x_1 \to \pm \infty} u_1(x_1,x_3) = \pm 1\,, \quad x_3 \in \bbR\,,
	\end{gathered}
\end{equation}
where the nonlocal operator $\cL$ has the Fourier symbol
\begin{equation}\label{eq:1DFourierSymbolCase1}
	\wt{m}(k) :=  \frac{2 \mu r_2 \big( p k_1^2 + k_3^2 \big) }{r_1 r_2 - \nu k_1^2}\,
\end{equation}
with $p$  defined in \eqref{eq:Defns:pAndQ}.
This symbol has a ``doubly nonlocal" nature, consisting of products of square roots of anisotropic symbols.
Since $\nu < 1/2$ we have
\begin{equation*}
	2 \mu \min\{ 1, p \} r_2(k) \leq \wt{m}(k) \leq \mu C(\delta,\nu) r_2(k) \qquad \text{ for all } k \in \bbR^2\,.
\end{equation*}
We now derive the integral formulation of $\cL$ in $\scS(\bbR^2)$. The integral formulation for $\cL$ agrees with the Fourier inversion formula whenever the function it acts on is sufficiently smooth, for instance $u \in \dot{H}^s(\bbR^2) \cap L^{\infty}(\bbR^2)$ for $s \geq 1$.

\begin{proposition}\label{prop:Aniso:1DOperator:IntegralForm}
	Suppose that $K_1$ and $K_2$ are the unique $C^{\infty}$ solutions in $\bbR^2 \setminus \{0\}$ of the partial differential equations
	\begin{equation}\label{eq:Aniso:KernelEquation1}
		(\delta(1-\nu^2) \p_1^4 + (1+\delta) \p_1^2 \p_3^2 + \p_3^4) K_1(x) = (\p_1^2 + \p_3^2)(p \p_1^2 + \p_3^2) \left( \frac{ 1 }{ ( x_1^2 + \delta x_3^2 )^{3/2} } \right)
	\end{equation}
	and
	\begin{equation}\label{eq:Aniso:KernelEquation2}
		(\delta(1-\nu^2) \p_1^4 + (1+\delta) \p_1^2 \p_3^2 + \p_3^4) K_2(x) = (p \p_1^4 + \p_1^2 \p_3^2)  \left( \frac{ 1 }{ ( x_1^2 + x_3^2 )^{3/2} } \right)\,.
	\end{equation}
	Then integral form of the operator $\cL$ defined in \eqref{eq:OneDimEqnForU1} is
	\begin{equation}\label{eq:1DOperator:TrueAnisotropic1}
		\cL w(x) = - \frac{1}{4 \pi} \int_{\bbR^2} \big( w(x-y)-w(x+y) - 2 w(x) \big) K(y) \, \rmd y\,,
	\end{equation}
	where $K(x) := 2 \mu \delta  \big(  \sqrt{\delta} K_1(x) + \nu K_2(x) \big)$.
\end{proposition}

\begin{proof}
We want the denominator of $\wt{m}$ to be a local operator so that we can integrate by parts in the half-Laplacian definition. If we  rationalize the denominator, then $\wt{m}$ remains in $\scS'(\bbR^2)$, and the Fourier symbol becomes
\begin{equation*}
	 \wt{m}(k) = 2 \mu \delta \frac{m_n(k) }{ m_d(k) }\,,
\end{equation*}
where the symbols $m_{n}(k)$ and $m_d(k)$ are defined as
\begin{equation*}
	\begin{split}
		m_n(k) &:=  r_2 \big( p k_1^2 + k_3^2 \big) \big( r_1 r_2 + \nu k_1^2 \big) \\
	\end{split}
\end{equation*}
and
\begin{equation}\label{eq:LdSymbol}
	m_d(k) := c k_1^4 + b k_1^2 k_3^2 + k_3^4 \quad \text{ with } \quad c := \delta(1-\nu^2)\,,\, b:=(1+\delta)\,.
\end{equation}
Here $m_d$ is a local operator, and $m_n$ is a sum of compositions of local operators with square-root operators.

Separate $m_n$ into two pieces: $m_n = m_{n,1} + \nu m_{n,2}$ where
\begin{equation*}
	\begin{split}
		m_{n,1} &:= r_2^2 \big( p k_1^2 + k_3^2 \big) r_1\,, \\
		m_{n,2} &:=  k_1^2 \big( p k_1^2 + k_3^2 \big) r_2\,.
	\end{split}
\end{equation*}
Note that, as functions of variables $k_1^2$ and $k_2^2$, both $\frac{m_{n,1}}{r_1}$ and $\frac{m_{n,2}}{r_2}$ are homogeneous second-order polynomials, hence define anisotropic biharmonic-type differential operators.

Define $L_n$, $L_{n,1}$ and $L_{n,2}$ to be the operators whose Fourier symbols are $m_n$, $\frac{m_{n,1}}{r_1}$ and $\frac{m_{n,2}}{r_2}$ respectively.
Therefore for any $w \in \scS(\bbR^2)$
$$
\cF( L_{n}w ) = \cF ( L_{n,1} \circ (-\Delta_{\frac{1}{\delta}})^{1/2} w + \nu  L_{n,2} \circ (-\Delta)^{1/2} w ) = (m_{n,1}(k) + \nu m_{n,2}(k) ) \hat{w} = m_n(k) \hat{w}
$$
in $\scS'(\bbR^2)$, where the anisotropic fractional Laplacian has been defined in \Cref{thm:AnisoLaplaceRep}.

Define $L_d$ to be the operator whose symbol is $m_d$. By using a cutoff function near the origin and using the dominated convergence theorem, we may assume that $\hat{w}$ vanishes near the origin. Now define $v \in \scS(\bbR^2)$ by $w = L_d v$. Thus $\hat{v}  = \frac{\hat{w}}{m_d(k)}$.
So, we have 
\begin{equation*}
	\begin{split}
		\cL w &:=  2 \mu \delta  \big( \cL_1 w + \nu \cL_2 w \big) \\
		&= 2 \mu \delta \left( (-\Delta_{\frac{1}{\delta}})^{1/2} L_{n,1} L_d^{-1} w 
		+ \nu (-\Delta)^{1/2} L_{n,2} L_d^{-1} w \right) \\
		&= 2 \mu \delta \left( (-\Delta_{\frac{1}{\delta}})^{1/2} L_{n,1} v 
		+ \nu (-\Delta)^{1/2} L_{n,2} v \right)\,.
	\end{split}
\end{equation*}
Treat $\cL_1$ and $\cL_2$ separately.
By the representation formula in Theorem \ref{thm:AnisoLaplaceRep}
\begin{equation*}
	\begin{split}
		\cL_1 u(x) &=  -\frac{\sqrt{\delta}}{4 \pi} \intdm{\bbR^2}{ \frac{   (L_{n,1})_x  v(x-y)+ (L_{n,1})_x  v(x+y)- 2 (L_{n,1})_x  v(x) }{ ( y_1^2 + \delta y_3^2 )^{3/2} } }{y} \\
		&= -\frac{\sqrt{\delta}}{4 \pi} \intdm{\bbR^2}{ \frac{   (L_{n,1})_y \left(   v(x-y)+ v(x+y) - \big( \sum_{|a| \leq 4} \frac{D^{a}v(x)}{a !} ( y^{a} + (-y)^{a} )  \big) \right) }{ (y_1^2 + \delta y_3^2 )^{3/2} } }{y}\,.
	\end{split}
\end{equation*}
Here we are using multi-index notation: $a = (a_1,a_3) \in \bbN^2_0$, $|a| = a_1 + a_3$, $a! = a_1 ! a_3!$, $y^a = y_1^{a_1} y_3^{a_3}$, and $D^a v(x) = \frac{\p^{|a|}}{\p_{x_1}^{a_1} \p_{x_3}^{a_3}} v(x)$.
Now integrating by parts (we can use integration by parts here without difficulty by using the P.V. definition of the integral operator and observe that the boundary term $\to 0$ as the P.V. limit is taken), 
\begin{equation*}
	\begin{split}
		&\cL_1 u(x) \\
		&= -\frac{\sqrt{\delta}}{4 \pi} \intdm{\bbR^2}{ \left(   v(x-y)+ v(x+y) -  \sum_{|a| \leq 4} \frac{D^{a}v(x)}{a !} ( y^{a} + (-y)^{a} )   \right)  (L_{n,1})_y  \left( \frac{ 1 }{ ( y_1^2 + \delta y_3^2 )^{3/2} } \right) }{y}\,.
	\end{split}
\end{equation*}

Since $K_1(y)$ satisfies
\begin{equation*}
	L_d K_1(y) = L_{n,1}  \left( \frac{ 1 }{ ( y_1^2 + \delta y_3^2 )^{3/2} } \right),
\end{equation*}
integrating by parts again gives
\begin{equation*}
	\begin{split}
		\cL_1 u(x) &= -\frac{\sqrt{\delta}}{4 \pi } \intdm{\bbR^2}{ \left(   v(x-y)+ v(x+y) -  \sum_{|a| \leq 4} \frac{D^{a}v(x)}{a !} ( y^{a} + (-y)^{a} )   \right)  (L_{d})_y  K_1(y) }{y} \\
		&= -\frac{\sqrt{\delta}}{4 \pi } \intdm{\bbR^2}{ (L_d)_y \left(   v(x-y)+ v(x+y) -  \sum_{|a| \leq 4} \frac{D^{a}v(x)}{a !} ( y^{a} + (-y)^{a} )   \right)  K_1(y) }{y} \\
		&= -\frac{\sqrt{\delta}}{4 \pi } \intdm{\bbR^2}{ \left(   (L_d)_x v(x-y) + (L_d)_x v(x+y) - 2 (L_d)_x v(x) \right)  K_1(y) }{y} \\
		&= -\frac{\sqrt{\delta}}{4 \pi } \intdm{\bbR^2}{ \left(   u(x-y) + u(x+y) - 2 u(x) \right)  K_1(y) }{y}\,.
	\end{split}
\end{equation*}

We treat $\cL_2$ the same way:
\begin{equation}\label{eq:LnSymbol}
	\begin{split}
		\cL_2 u(x) &=  -\frac{1}{4 \pi} \intdm{\bbR^2}{ \frac{   (L_{n,2})_x  v(x-y)+ (L_{n,2})_x  v(x+y)- 2 (L_{n,2})_x  v(x) }{ ( y_1^2 + y_3^2 )^{3/2} } }{y} \\
		&= -\frac{1}{4 \pi} \intdm{\bbR^2}{ \frac{   (L_{n,2})_y \left(   v(x-y)+ v(x+y) -  \sum_{|a| \leq 4} \frac{D^{a}v(x)}{a !} ( y^{a} + (-y)^{a} )   \right) }{ (y_1^2 + y_3^2 )^{3/2} } }{y}\,.
	\end{split}
\end{equation}
Integrating by parts,
\begin{equation*}
	\begin{split}
		&\cL_2 u(x) \\
		&= -\frac{1}{4 \pi} \intdm{\bbR^2}{ \left(   v(x-y)+ v(x+y) -  \sum_{|a| \leq 4} \frac{D^{a}v(x)}{a !} ( y^{a} + (-y)^{a} ) \right)  (L_{n,2})_y  \left( \frac{ 1 }{ ( y_1^2 + y_3^2 )^{3/2} } \right) }{y}\,.
	\end{split}
\end{equation*}
Since $K_2(y)$ satisfies
\begin{equation*}
	L_d K_2(y) = L_{n,2}  \left( \frac{ 1 }{ ( y_1^2 + y_3^2 )^{3/2} } \right),
\end{equation*}
integrating by parts again results in
\begin{equation*}
	\begin{split}
		\cL_2 u(x) &= -\frac{1}{4 \pi } \intdm{\bbR^2}{ \left(   v(x-y)+ v(x+y) - \sum_{|a| \leq 6} \frac{D^{a}v(x)}{a !} ( y^{a} + (-y)^{a} )   \right)  (L_{d})_y  K_2(y) }{y} \\
		&= -\frac{1}{4 \pi } \intdm{\bbR^2}{ (L_d)_y \left(   v(x-y)+ v(x+y) -  \sum_{|a| \leq 6} \frac{D^{a}v(x)}{a !} ( y^{a} + (-y)^{a} )  \right)  K_2(y) }{y} \\
		&= -\frac{1}{4 \pi } \intdm{\bbR^2}{ \left(   (L_d)_x v(x-y) + (L_d)_x v(x+y) - 2 (L_d)_x v(x) \right)  K_2(y) }{y} \\
		&= -\frac{1}{4 \pi } \intdm{\bbR^2}{ \left(   u(x-y) + u(x+y) - 2 u(x) \right)  K_2(y) }{y}\,.
	\end{split}
\end{equation*}
\end{proof}

\subsection{The Formula for the Kernel}

We now find the formula for the nonlocal kernel $K$ in \eqref{eq:1DOperator:TrueAnisotropic1} by solving the PDEs for $K_1(x)$ and $K_2(x)$.

\begin{lemma}\label{lma:Aniso:Kernel1Properties}
	Let $\mu$, $\delta$ and $\nu$ satisfy \eqref{eq:Perp:UniformEllipticity:NuAndDelta}, define $p$ and $q$ as in \eqref{eq:Defns:pAndQ}, and define $b$ and $c$ as in \eqref{eq:LdSymbol}.	
	Then there exists a solution $K_1$ to the equation \eqref{eq:Aniso:KernelEquation1} 
	given by
	\begin{equation}\label{eq:Kernel1Ansatz}
		K_1(x_1,x_3) = \frac{A x_1^{12} + B x_1^{10} x_3^2 + C x_1^{8} x_3^4 + D x_1^6 x_3^6 + E x_1^4 x_3^8 + F x_1^2 x_3^{10} + G x_3^{12} }{(x_1^2 + \delta x_3^2)^{\frac{3}{2}} (x_1^4 + b x_1^2 x_3^2 + c x_3^2)^3 }\,,
	\end{equation}
	where
	\begin{equation}\label{eq:Kernel1Coefficients}
		\begin{split}
			A &= \frac{-2 b+\delta +2 p+2}{\delta } \,, \\
			B &= \frac{3 \left(2 b^2-2 b (\delta +p+1)-4 c+3 \delta +3 \delta  p+4 p\right)}{\delta }\,, \\
			C &= \frac{9 b^2 \delta +6 c (3 b-5 \delta -4 p-4)-6 b \left(\delta ^2+\delta +\delta  p-p\right)+3 \delta  (2 \delta +2 \delta  p+11 p)}{\delta }\,, \\
			D &= \frac{b^2 (\delta  (2 \delta +1)+(\delta +2) p)-2 b (c (-13 \delta +p+1)+\delta  (\delta +(\delta -13) p))}{\delta} \\
			&\qquad + \frac{20 c^2-2 c (\delta  (10 \delta +23)+(23 \delta +10) p)+20 \delta ^2 p}{\delta }\,, \\
			E &= \frac{-6 c \left(b \left(-\delta ^2+\delta +\delta  p+p\right)+\delta  (4 \delta +4 \delta  p+5 p)\right)+9 b \delta  p (b+2 \delta )+c^2 (33 \delta +6 p+6)}{\delta }\,, \\
			F &=  6 b^2 \delta  p-6 c (b (\delta +\delta  p+p)+2 \delta  p)+3 c^2 (4 \delta +3 p+3)\,, \\
			G &= c (c (2 \delta +2 \delta  p+p)-2 b \delta  p)\,.
		\end{split}
	\end{equation}
	$K_1$ is the unique solution of \eqref{eq:Aniso:KernelEquation1} in $C^{\infty}(\bbR^2 \setminus \{ 0\})$, and $K_1$ satisfies the following:
	\begin{enumerate}
		\item[i)] $K_1(x) = K_1(-x)$, $K_1(\rho x) = \rho^{-3} K_1(x)$ for $\rho > 0$.
		\item[ii)] $|x^{a} D^a K_1(x)| < \frac{C}{|x|^3}$ for any multi-index $a \in \bbN_0^2$ and any $x \neq 0$.
	\end{enumerate}
\end{lemma}

\begin{proof}
Recall the definitions of the operators $L_d$ and $L_{n,1}$ from the proof of \Cref{prop:Aniso:1DOperator:IntegralForm}.
The right-hand side of \eqref{eq:Aniso:KernelEquation1} is equal to
\begin{equation*}
	G_1(x_1,x_3) := \frac{P_1(x_1, x_3)}{ ( x_1^2 + \delta x_3^2 )^{\frac{11}{2}} }\,,
\end{equation*}
where $P_1$ is the fourth-degree polynomial
\begin{equation*}
	\begin{split}
		P_1(x_1,x_3) &:= 45
		\Big(  
		d_1(\nu,\delta) x_1^4 + 
		d_2(\nu,\delta) x_1^2 x_3^2 + 
		d_3(\nu,\delta) x_3^4 \Big)\,, \text{ with } \\
		d_1(\nu,\delta) &:= 8 p - 2 \delta (1+p) + \delta^2 \,, \\
		d_2(\nu,\delta) &:= \delta (-12 p + 17 (1+p) \delta - 12 \delta^2 )\,, \\
		d_3(\nu,\delta) &:= \delta^2 (p - \delta 2(1+p) + 8 \delta^2) \,. \\
	\end{split}
\end{equation*}
Our guess for a solution of \eqref{eq:Aniso:KernelEquation1} is \eqref{eq:Kernel1Ansatz}, with constants $A$ through $G$ to be determined.
Applying $L_d$ to this ansatz, we obtain a function of the form $K_1(x_1,x_3) = \frac{\Pi(x_1,x_3)}{ ( x_1^2 + \delta x_3^2 )^{\frac{11}{2}} (x_1^4 + b x_1^2 x_3^2 + c x_3^2)^5 }$, where $\Pi$ is a polynomial with terms $x_1^{2i} x_3^{2(12-i)}$, where $i \in \{0,1,\ldots, 12 \}$. Multiplying the equation $L_d K_1(x) = G_1(x)$ by the denominator of this expression, we obtain a linear system of thirteen equations for the seven unknown coefficients $A$ through $G$.
This linear system has a unique solution 
and we obtain that $K_1$ is of the form \eqref{eq:Kernel1Ansatz} with coefficients given by \eqref{eq:Kernel1Coefficients}.

Properties i) and ii) for $K_1(x)$ are now apparent. Converting to polar coordinates $(x_1,x_3) = (r\cos\theta,r\sin\theta)$,
the equation \eqref{eq:Aniso:KernelEquation1} becomes a fourth-order ODE in $\theta$ after multiplying by $r^{11}$, with coefficients that are polynomials in $\sin(\theta)$ and $\cos(\theta)$ with bounded right-hand side. We have the existence and uniqueness of a solution to this ODE with $\pi$-periodic boundary conditions on its derivatives up to order $3$, and since we have also shown that $K$ is smooth away from the origin, we deduce that $K$ is unique.
\end{proof}

\begin{lemma}\label{lma:Aniso:Kernel2Properties}
	Let $\mu$, $\delta$ and $\nu$ satisfy \eqref{eq:Perp:UniformEllipticity:NuAndDelta}. Define $p$ and $q$ as in \eqref{eq:Defns:pAndQ}, and define $b$ and $c$ as in \eqref{eq:LdSymbol}.
	Then there exists a solution $K_2$ to the equation \eqref{eq:Aniso:KernelEquation2} 
	given by
	\begin{equation}\label{eq:Kernel2Ansatz}
		K_2(x_1,x_3) = \frac{A x_1^{12} + B x_1^{10} x_3^2 + C x_1^{8} x_3^4 + D x_1^6 x_3^6 + E x_1^4 x_3^8 + F x_1^2 x_3^{10} + G x_3^{12} }{(x_1^2 + x_3^2)^{\frac{3}{2}} (x_1^4 + b x_1^2 x_3^2 + c x_3^2)^3 }\,.
	\end{equation}
	where
	\begin{equation}\label{eq:Kernel2Coefficients}
		\begin{split}
			A &= 2 \,, \\
			B &= -6 b+12 p+9\,, \\
			C &= 6 b (p-1)-24 c+33 p+6\,, \\
			D &= b^2-2 b (c+1)+2 (b+13) b p-20 c p-46 c+20 p\,, \\
			E &= -6 c ((b+5) p+b+4)+9 b (b+2) p+6 c^2\,, \\
			F &=  6 b^2 p-6 b c (p+1)+3 c (3 c-4 p)\,, \\
			G &= c (c (p+2)-2 b p) \,.
		\end{split}
	\end{equation}
	$K_2$ is the unique solution of \eqref{eq:Aniso:KernelEquation2} in $C^{\infty}(\bbR^2 \setminus \{ 0\})$, and $K_2$ satisfies the following:
	\begin{enumerate}
		\item[i)] $K_2(x) = K_2(-x)$, $K_2(\rho x) = \rho^{-3} K_2(x)$ for $\rho > 0$.
		\item[ii)] $|x^{a} D^a K_2(x)| < \frac{C}{|x|^3}$ for any multi-index $a \in \bbN_0^2$ and any $x \neq 0$.
	\end{enumerate}
\end{lemma}

\begin{proof}	
The right-hand side of \eqref{eq:Aniso:KernelEquation2} is equal to
\begin{equation*}
	G_2(x_1,x_3) := \frac{P_2(x_1, x_3)}{ ( x_1^2 + x_3^2 )^{\frac{11}{2}} }\,,
\end{equation*}
where $P_2$ is a fourth-degree polynomial defined as
\begin{equation*}
	\begin{split}
		P_2(x_1,x_3) := 45 \left( (8 p-2) x_1^4 + (17-12 p) x_1^2 x_3^2+(p-2) x_3^4 \right)\,.
	\end{split}
\end{equation*}
Our guess for a solution of \eqref{eq:Aniso:KernelEquation2} is \eqref{eq:Kernel2Ansatz}.
In exactly the same way, we apply $L_d$ to this ansatz and obtain a function of the form $K_2(x_1,x_3) = \frac{\Pi(x_1,x_3)}{ ( x_1^2 + x_3^2 )^{\frac{11}{2}} (x_1^4 + b x_1^2 x_3^2 + c x_3^2)^5 }$, where $\Pi$ is a polynomial with terms $x_1^{2i} x_3^{2(12-i)}$, where $i \in \{0,1,\ldots, 12 \}$. Multiplying the equation $L_d K_2(x) = G_2(x)$ by the denominator of this expression, we obtain a linear system of thirteen equations for the seven unknown coefficients $A$ through $G$.
This linear system also has a unique solution,
and we obtain that $K_2$ is of the form \eqref{eq:Kernel2Ansatz} with $A$ through $G$ given by \eqref{eq:Kernel2Coefficients}.
\end{proof}

In summary, we have
\begin{proposition}\label{prop:KernelForm:perp:u1Dep}
	The kernel function $K : \bbR^2 \setminus \{ 0\} \to \bbR$ of the operator \eqref{eq:1DOperator:TrueAnisotropic1} is defined by
	\begin{equation}\label{eq:perp:u1Eqn:Kernel}
		K(x) := 2 \mu \delta  \big(  \sqrt{\delta} K_1(x) + \nu K_2(x) \big)\,.
	\end{equation}
	Here the functions $K_1(x)$ and $K_2(x)$ are defined via \eqref{eq:Kernel1Ansatz}-\eqref{eq:Kernel1Coefficients} and \eqref{eq:Kernel2Ansatz}-\eqref{eq:Kernel2Coefficients}, respectively.
\end{proposition}

\begin{remark}
	When $\delta = 1$, then $p = 1$, $b = 2$ and $c = (1+\nu)(1-\nu)$, and so the formula for $K$ simplifies to
	\begin{equation}\label{eq:Kernel:FullIsotropic}
		K(z) = 2 \mu \frac{A z_1^4 + Bz_1^2 z_3^2 + C z_3^4}{|z| ( z_1^2 + q z_3^2)^{3}}
	\end{equation}
	with
	\begin{equation}\label{eq:Coefficients:FullIsotropic}
		\begin{gathered}
			A = 3 - 2 q, \quad B= 2(3 q^2 - 5 q + 3), \quad C= q( 3q-2),
		\end{gathered}
	\end{equation}
	and $q = 1 - \nu$. This is exactly the kernel associated to the 1D reduced operator in the case of isotropic elasticity \cite{dong2021existence}; see also \Cref{sec:parallel} below.
\end{remark}

\subsection{The Region Where $K$ is Positive}

A crucial component of our analysis requires that $K$ is positive everywhere in $\bbR^2 \setminus \{0\}$. Below, for case I, we compute the range of elastic coefficients for which this holds.

\begin{proposition}[Positivity of kernel for case I]\label{prop:temp:u1}
	Let $\mu$, $\delta$ and $\nu$ satisfy \eqref{eq:Perp:UniformEllipticity:NuAndDelta}. Define $p$ and $q$ as in \eqref{eq:Defns:pAndQ}, and define $b$ and $c$ as in \eqref{eq:LdSymbol}.
	Suppose also that $\nu$ and $\delta$ satisfy \eqref{eq:Perp:u1Dependence:KernelPosRegion}.  Let $K(x)= 2 \mu \delta  \big(  \sqrt{\delta} K_1(x) + \nu K_2(x) \big)$ be the kernel for the nonlocal operator in \eqref{eq:nonlocalMain1d}, where functions $K_1(x)$ and $K_2(x)$ are defined via \eqref{eq:Kernel1Ansatz}-\eqref{eq:Kernel1Coefficients} and \eqref{eq:Kernel2Ansatz}-\eqref{eq:Kernel2Coefficients}, respectively. Then there exist positive constants $c_{\delta,\nu}$ and $C_{\delta,\nu}$ depending only on $\delta$ and $\nu$ such that $\frac{\mu c_{\delta,\nu}}{|x|^3} \leq K(x) \leq \frac{\mu C_{\delta,\nu}}{|x|^3}$.
\end{proposition}

\begin{proof}
Since $K(\rho x) = \rho^{-3} K(x)$ for any $\rho > 0$, we need only be concerned with the minimum values of $K$ on $|x|=1$, i.e. the minimum values of $K(\cos\theta,\sin\theta)$ for $\theta \in [0,\pi)$.

Based on the plots of the exact expression for $K$ for varying $\nu$ and $\delta$, the most likely candidates for $\arg \min_{\theta \in [0,\pi)} K(\cos\theta,\sin\theta)$ are $\theta =0 $ and $\theta = \pi/2$. 
Therefore, the values of $(\nu,\delta)$ for which $K(x) > 0$ is at most those for which both $K(1,0) > 0$ and $K(0,1) > 0$.
Plugging these two values in gives
\begin{equation*}
	K(1,0) = \delta ^{3/2} (3-4 \nu )+\delta ^{5/2} (4 \nu -2)+2 \delta  \nu\,, \qquad K(0,1) = \frac{\delta  (\nu  (-2 \delta  \nu +\delta +2 \nu -4)+1)}{(\nu -1)^2}\,.
\end{equation*}
Both of these quantities are positive exactly when $\nu$ and $\delta$ satisfy \eqref{eq:Perp:u1Dependence:KernelPosRegion}.
Because of the explicit formula for $K$, we claim that \eqref{eq:Perp:u1Dependence:KernelPosRegion} precisely describes the range of coefficients for which $K$ is positive. It can be checked that the boundary of \Cref{fig:PosReg} coincides numerically with the boundary of the set $\{ (\nu,\delta) \, : \, \eqref{eq:Perp:u1Dependence:KernelPosRegion} \text{ is satisfied} \}$.
\end{proof}

\begin{figure}[h]
	\centering
		\centering
		\includegraphics[width=.6\textwidth]{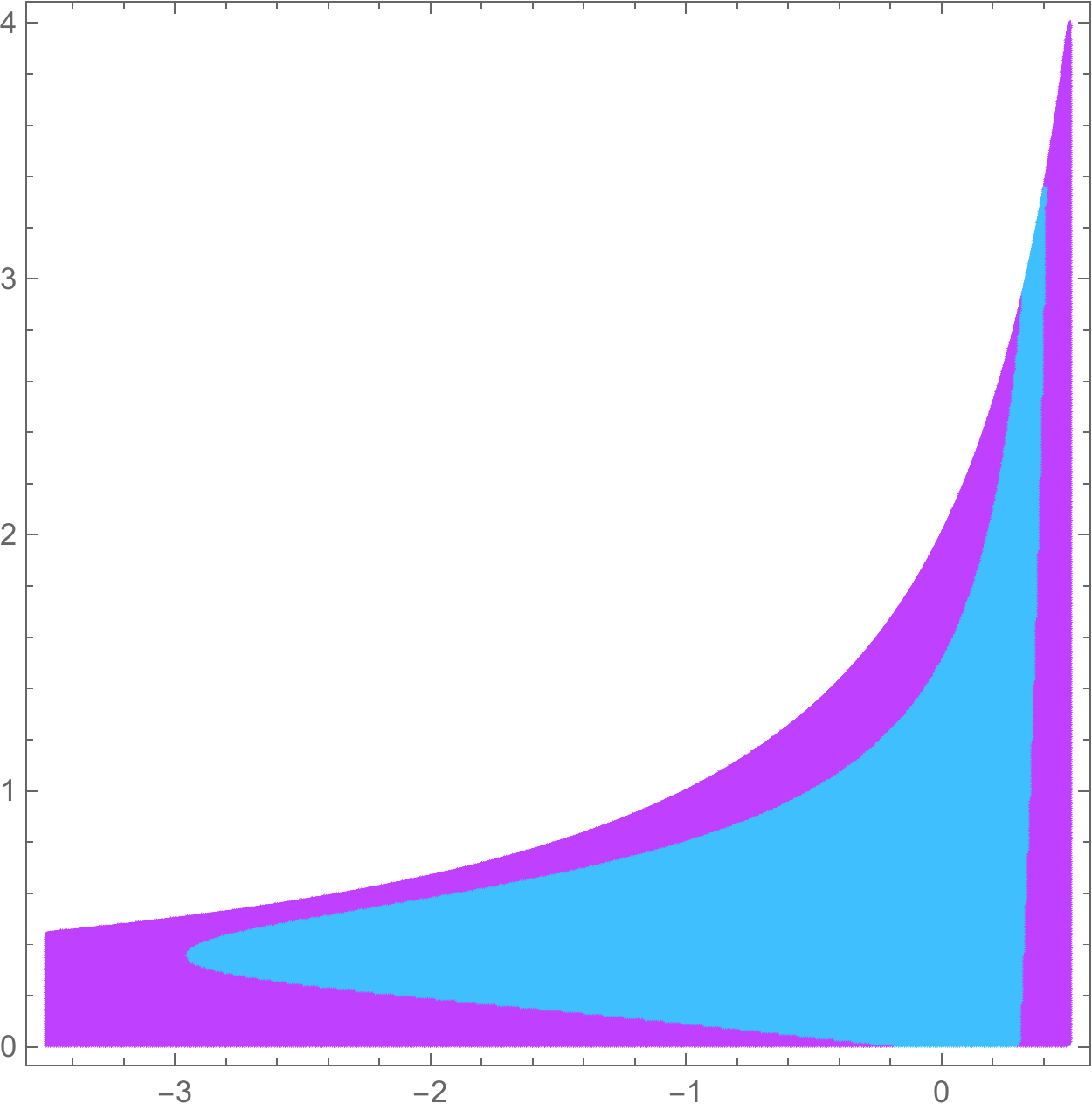}
		\caption{Plot of the numerical minima of $K(\cos\theta,\sin\theta)$ in the $(\nu,\delta)$-plane. The outer region consists of $(\nu,\delta)$ that satisfy \eqref{eq:Perp:UniformEllipticity:NuAndDelta}. The inner region consists of the $(\nu,\delta)$ for which the numerical minimum of $K(\cos\theta,\sin\theta)$ is positive.
		}
		\label{fig:PosReg}
\end{figure}

\begin{remark}
	When $\delta = 1$ and $K$ is given by \eqref{eq:Kernel:FullIsotropic}-\eqref{eq:Coefficients:FullIsotropic}, the first derivative test applied to $\theta \mapsto K(\cos\theta,\sin\theta)$ reveals that
	\begin{equation*}
		K(z) > 0 \text{ if and only if } \frac{2}{3} < q < \frac{3}{2}\,,
	\end{equation*}
	i.e. if and only if $\nu \in ( -\frac{1}{2}, \frac{1}{3} )$. This is the same range for $\nu$ found in \cite{dong2021existence} via variation of parameters.
\end{remark}

The energy corresponding to \eqref{eq:nonlocalMain1d} is
\begin{equation}\label{eq:nonlocalMain1d:Energy}
	E_{\Gamma}(u_1) := \frac{1}{2} \int_{\bbR^2} (\cL u_1) u_1 \, \rmd x + \int_{\bbR^2} W(u_1) \, \rmd x\,,
\end{equation}
with $E_{\Gamma}(u_1;B_R)$ defining the localized version. The goal now is to show existence, uniqueness, and properties of solutions to \eqref{eq:nonlocalMain1d}. Thanks to the properties of $K$, solutions to \eqref{eq:nonlocalMain1d} have a 1D profile, which we will prove in \Cref{sec:1DProfile}.

\section{The 1D reduced equation in the perpendicular case: $\p_{u_1} W = 0$}\label{sec:Perp:u3Eqn}
Assuming that $W(u_1,u_3) = W(u_3)$, we now derive analogously to \Cref{sec:Perp:u1Eqn} the reduced scalar equation for case II and then characterize the explicit form of the nonlocal kernel in Proposition \ref{prop:AlmostIso:1DOperator:IntegralForm}.
Since the material is anisotropic, making the assumption that the misfit potential depends only on $\hat{u}_3$ gives us a different reduced equation. 
If one makes this assumption, then proceeding the same way as we did in \Cref{sec:Perp:u1Eqn} we get an equation for $\hat{u}_3$:
\begin{equation}\label{eq:OneDimEqnForU3}
	\cF(\cL u_3)(k_1,k_3) := \frac{\det \bbA(k)}{a_{11}(k)} \hat{u}_3(k) = - \cF (W'(u_3))\,.
\end{equation}
Inverting the Fourier transform, we obtain the 1D nonlocal equation along with \eqref{eq:FarFieldCond:u3}:
\begin{equation}\label{eq:nonlocalMain1d:u3Eqn}
	\begin{gathered}
		\cL u_3 (x_1,x_3) = - W'(u_3(x_1,x_3))\,, \quad (x_1,x_3) \in \Gamma\,, \\
		\lim\limits_{x_3 \to \pm \infty} u_3(x_1,x_3) = \pm 1\,, \quad x_1 \in \bbR\,,
	\end{gathered}
\end{equation}
where
\begin{equation}\label{eq:AlmostIsotropic1DSymbol}
	\cF(\cL u_3)(k)  = \frac{2 \mu r_2 \big( \delta (2(1-\nu)-\delta(1-2\nu)) k_1^2 + k_3^2 \big) }{\delta (2(1-\nu)-\delta(1-2\nu)) k_1^2 + (1-\nu) k_3^2} \hat{u}_3(k)\,.
\end{equation}
This equation is similar to the 1D reduced equation in the case of isotropic elasticity. However, neither the numerator nor denominator are Laplacian operators.
Recalling the definition of $p$ and $q$ from \eqref{eq:Defns:pAndQ},
we can rewrite \eqref{eq:AlmostIsotropic1DSymbol} as
\begin{equation}\label{eq:OperatorSymbol:perp:u3Dep}
	\cF(\cL u_3)(k) =  \frac{2 \mu r_2 \big( p k_1^2 + k_3^2 \big) }{p k_1^2 + q k_3^2} \hat{u}_3(k)\,.
\end{equation}
Note that from definition in \eqref{eq:Defns:pAndQ}, $0 < p \leq 4$ and $1/2 < q < \infty$ for any $\nu$ and $\delta$ satisfying \eqref{eq:Perp:UniformEllipticity:NuAndDelta}, so we have
\begin{equation*}
	C(\mu,p,q)^{-1} r_2(k) \leq \cF(\cL u_3)(k) \leq C(\mu,p,q) r_2(k)  \text{ for all } k \in \bbR^2\,.
\end{equation*}
We use the notational conventions from \Cref{sec:Perp:u1Eqn} (e.g. $\cL$ has nonlocal kernel $K$, etc.) but the explicit definitions (i.e. the expressions for $\cL$ and $K$) are distinct from those in \Cref{sec:Perp:u1Eqn} unless indicated otherwise.

\begin{proposition}\label{prop:AlmostIso:1DOperator:IntegralForm}
	The integral form of $\cL$ defined in \eqref{eq:OneDimEqnForU3} is
	\begin{equation*}
		\cL w = - \frac{1}{4 \pi} \int_{\bbR^2} \big( w(x-y)-w(x+y) - 2 w(x) \big) K(y) \, \rmd y\,,
	\end{equation*}
	where $K$
	is the unique solution in $C^{\infty}(\bbR^2 \setminus \{0\})$ to the PDE
	\begin{equation}\label{eq:AlmostIso:KernelEqn}
		(p \p_1^2 +  q \p_3^2) K(z) = 2 \mu (p \p_1^2 + \p_3^2) \left[ \frac{1}{|z|^3} \right]\,.
	\end{equation}
\end{proposition}

\begin{proof}
	Define $L_n$ and $L_d$ to be the operators whose symbols are $m_n(k) := p k_1^2 + k_3^2$ and $m_d(k) := p k_1^2 + q k_3^2$ respectively. Let $w \in \scS(\bbR^2)$. By using a cutoff function near the origin and using the dominated convergence theorem, we may assume that $\hat{w}$ vanishes near the origin. Define $v \in \scS(\bbR^2)$ by $L_d v = w$, so that $\hat{v} = \frac{\hat{w}}{m_d(k)}$. Therefore we have
	\begin{equation*}
		\begin{split}
			\cL w &= 2 \mu \left( (-\Delta)^{1/2} L_{n} L_d^{-1} w \right) \\
			&=2 \mu \left( (-\Delta)^{1/2} L_{n} v \right).\,
		\end{split}
	\end{equation*}
	By the representation formula for $(-\Delta)^{1/2}$ in Theorem \ref{thm:AnisoLaplaceRep}
	\begin{equation*}
		\begin{split}
			\cL w(x) &=  -\frac{ 2\mu }{4 \pi} \intdm{\bbR^2}{ \frac{   (L_{n})_x  v(x-y)+ (L_{n})_x  v(x+y)- 2 (L_{n})_x  v(x) }{ ( y_1^2 + y_3^2 )^{3/2} } }{y} \\
			&= -\frac{2\mu}{4 \pi} \intdm{\bbR^2}{ \frac{   (L_{n})_y \left(   v(x-y)+ v(x+y) - \big( \sum_{|a| \leq 2} \frac{D^{a}v(x)}{a !} ( y^{a} + (-y)^{a} )  \big) \right) }{ (y_1^2 + y_3^2 )^{3/2} } }{y}\,.
		\end{split}
	\end{equation*}
	Here we are using multi-index notation. $a = (a_1,a_3) \in \bbN^2_0$, $|a| = a_1 + a_3$, $a! = a_1 ! a_3!$, $y^a = y_1^{a_1} y_3^{a_3}$, and $D^a v(x) = \frac{\p^{|a|}}{\p_{x_1}^{a_1} \p_{x_3}^{a_3}} v(x)$.
	Now integrating by parts (we can use integration by parts here without difficulty by using the P.V. definition of the integral operator and observe that the boundary term $\to 0$ as the P.V. limit is taken),
	\begin{equation*}
		\begin{split}
			&\cL w(x) \\
			&= -\frac{1}{4 \pi} \intdm{\bbR^2}{ \left(   v(x-y)+ v(x+y) -  \sum_{|a| \leq 2} \frac{D^{a}v(x)}{a !} ( y^{a} + (-y)^{a} )  \right)  (L_{n})_y  \left( \frac{ 2\mu }{ ( y_1^2 + y_3^2 )^{3/2} } \right) }{y}\,.
		\end{split}
	\end{equation*}
	Now, since $K(y)$ satisfies
	\begin{equation}\label{eq:AlmostIso:1DEqnPf1}
		L_d K(y) = L_{n}  \left( \frac{ 2\mu }{ |y|^{3} } \right)
	\end{equation}
	integrating by parts again gives
	\begin{equation*}
		\begin{split}
			\cL_1 w(x) &= -\frac{1}{4 \pi } \intdm{\bbR^2}{ \left(   v(x-y)+ v(x+y) - \big( \sum_{|a| \leq 2} \frac{D^{a}v(x)}{a !} ( y^{a} + (-y)^{a} )  \big) \right)  (L_{d})_y  K(y) }{y} \\
			&= -\frac{1}{4 \pi } \intdm{\bbR^2}{ (L_d)_y \left(   v(x-y)+ v(x+y) - \big( \sum_{|a| \leq 2} \frac{D^{a}v(x)}{a !} ( y^{a} + (-y)^{a} )  \big) \right)  K(y) }{y} \\
			&= -\frac{1}{4 \pi } \intdm{\bbR^2}{ \left(   (L_d)_x v(x-y) + (L_d)_x v(x+y) - 2 (L_d)_x v(x) \right)  K(y) }{y} \\
			&= -\frac{1}{4 \pi } \intdm{\bbR^2}{ \left(   w(x-y) + w(x+y) - 2 w(x) \right)  K(y) }{y}\,.
		\end{split}
	\end{equation*}
\end{proof}

\begin{proposition}[Homogeneity and positivity of kernel for case II]\label{lma:AlmostIso:KernelProperties}
	Let $\delta$ and $\nu$ satisfy \eqref{eq:Perp:UniformEllipticity:NuAndDelta}. Define $p$ and $q$ as in \eqref{eq:Defns:pAndQ}.
	Let the kernel $K : \bbR^2 \setminus \{0\} \to \bbR$ be defined as 
	\begin{equation}\label{eq:AlmostIso:KernelDefn1}
		K(z) := 2 \mu \frac{A z_1^6 + Bz_1^4 z_3^2 + C z_1^2 z_3^4 + D z_3^6}{|z|^3 ( q z_1^2 + p z_3^2)^{3}}
	\end{equation}
	with
	\begin{equation}\label{eq:AlmostIso:KernelDefn2}
		\begin{gathered}
			A= 2 p q^2-2 p q+q^2, \quad B= -6 p^2 q+6 p^2+9 p q^2-6 p q,\\ 
			C= -6 p^2 q+9 p^2+6 p q^2-6 p q,\quad D = p^3-2 p^2 q+2 p^2\,.
		\end{gathered}
	\end{equation}
	Then $K$ satisfies the following:
	\begin{enumerate}
		\item[i)] $K(x) = K(-x)$, $K(\rho x) = \rho^{-3} K(x)$ for $\rho > 0$.
		\item[ii)] $|x^{a} D^a K(x)| < \frac{C}{|x|^3}$ for any multi-index $a \in \bbN_0^2$ and any $x \neq 0$.
		\item[iii)] $K$ is the unique solution in $C^{\infty}(\bbR^2 \setminus \{0\})$ to the equation \eqref{eq:AlmostIso:KernelEqn} that is homogeneous of order $-3$.
		\item[iv)] Suppose that $\nu$ and $\delta$ satisfy \eqref{eq:Perp:u3Dependence:KernelPosRegion}. Then for each $(\nu,\delta) \in \cU^+$ there exist positive constants $c_{\delta,\nu}$ and $C_{\delta,\nu}$ depending only on $\delta$ and $\nu$ such that $\frac{\mu c_{\delta,\nu}}{|x|^3} \leq K(x) \leq \frac{\mu C_{\delta,\nu}}{|x|^3}$.
	\end{enumerate}
\end{proposition}

\begin{proof}
	Property i) is obvious, and ii) is easily derived from i). Property iii) can be seen by converting to polar coordinates $(x_1,x_3) = (r\cos\theta,r\sin\theta)$;
	after multiplying by $r^3$ the equation \eqref{eq:AlmostIso:KernelEqn} becomes a second-order ODE with bounded $C^{\infty}$ coefficients in $\theta$
	with $\pi$-periodic boundary conditions on $\theta \mapsto K(\cos\theta,\sin\theta)$.
	Thus there exists a unique twice-differentiable solution, and that the solution must be $K(\cos\theta,\sin\theta)$ follows from property ii) and from the fact that $K(x)$ solves \eqref{eq:AlmostIso:KernelEqn} in $\bbR^2 \setminus \{0\}$.
	
	It remains to prove property iv). Using the previous discussion, it will suffice to examine the critical points of the function $\overline{K}(\theta) := K(\cos\theta,\sin\theta)$ in the interval $[0,\pi/2]$. Using elementary algebra, the roots of the function $\overline{K}'(\theta)$ in the interval $[0,\pi/2]$ are $0$, $\pi/2$, and $\arccos(\zeta^{\pm})$, where
	\begin{equation*}
		\zeta^{\pm} := \sqrt{\frac{2q^2 - pq - p^2 \pm  \sqrt{(p-q)^2 \left(p^2+q^2\right)}}{(p-q)^2}}\,.
	\end{equation*}
	In the range $0 < p \leq 4$ and $\frac{1}{2} < q < \infty$, $\zeta^+ \in [-1,1]$ only for $p \geq \frac{4}{3}q$, and $\zeta^- \in [-1,1]$ is a real number only for $p \leq \frac{3}{4}q$.
	Therefore, the equation $\overline{K}'(\theta) = 0$ has exactly two solutions for $\frac{3}{4} q \leq p \leq \frac{4}{3} q$, and exactly three solutions otherwise. 
	
	Therefore $K(x_1,x_3) >0$ if and only if $\overline{K}(\theta) > 0$ if and only if
	\begin{equation*}
		\overline{K}(0) >0\,, \quad \overline{K}(\pi/2) > 0\,, \quad \text{ and } \quad \overline{K}(\arccos(\zeta^\pm)) > 0\,.
	\end{equation*}
	Substituting directly gives the relations
	\begin{equation}\label{eq:AlmostIso:Pf1}
		\begin{split}
			\frac{2 (2 p q- 2 p +q)}{q^2} &>0\,, \\
			\frac{2 (p-2 q+2)}{p} &>0\,, \\
			\frac{(q-1) \left(p^3+q^3+\left(p^2+q^2\right) \sqrt{p^2+q^2}\right) + 4 (1-p) p q^2 }{2 p (q-p) q^2} &> 0\text{ when } p \leq \frac{3}{4} q \text{ or }p \geq \frac{4}{3} q\,.
		\end{split}
	\end{equation}
	The region $\cU^+$ is then defined by these three conditions along with the ellipticity conditions \eqref{eq:Perp:UniformEllipticity:NuAndDelta}. It can be seen after elementary algebraic manipulations that the first condition gives $\nu > \frac{\delta(\delta-2)}{2(\delta^2-\delta+1)}$ and the second gives $\nu < \frac{2}{ 1 + 4 \delta - 2 \delta^2 + \sqrt{1 - 8 \delta + 28 \delta^2 - 16 \delta^3 + 4 \delta^4} }$.
	To see that the third condition leads to $p > \wt{r}$, we use a combination of analytic and numeric justification.	
	First, when $p \geq 4q/3$ it is easy to see that the expression $\overline{K}(\arccos(\zeta^\pm))$ is always positive. Second, we note that $\overline{K}(\arccos(\zeta^{\pm})) \leq \min \{ \overline{K}(0), \overline{K}(\pi/2) \}$ in the region $U$ defined by
	$$
	U := \{ (\nu,\delta) \, : \eqref{eq:Perp:UniformEllipticity:NuAndDelta} \text{ is satisfied and } \, 1/2 < q < 1 , p \leq 3q/4 \}\,,
	$$
	and $\overline{K}(\arccos(\zeta^{\pm})) \geq \min \{ \overline{K}(0), \overline{K}(\pi/2) \}$ in the region $\{ (\nu,\delta) \, : \eqref{eq:Perp:UniformEllipticity:NuAndDelta} \text{ is satisfied and } 1 \geq q , p \leq 3q/4 \}$.
	Third, the conjugate
	\begin{equation*}
		\frac{(q-1) \left(p^3+q^3-\left(p^2+q^2\right) \sqrt{p^2+q^2}\right) + 4 (1-p) p q^2 }{2 p (q-p) q^2}
	\end{equation*}
	is positive for all $(p,q) \in U$.
	
	Multiplying the third expression in \eqref{eq:AlmostIso:Pf1} by the conjugate, we obtain a rational expression whose denominator is positive and whose numerator is the polynomial
	\begin{multline*}
		x \mapsto -8  (-1 + q) x^4 + 8 (-1 + q) q^3 \\
		+  q^2 (13 + 14 q - 11 q^2) x + (-11 + 14 q + 13 q^2) x^3 + 
		2 q (1 - 18 q + q^2) x^2 \,.
	\end{multline*}
	The roots of this polynomial can be found exactly (although their analytic expressions are cumbersome to work with) and it can be verified numerically that all four roots are real for all $(p,q) \in U$.
	Putting all these observations together, we see that $\overline{K}(\arccos(\zeta^{\pm})) > 0$ for $(p,q) \in U$ exactly when $p$ is larger than either two or all of the roots of the polynomial.
	The following facts can be verified numerically: For $(p,q) \in U$, one root is always negative. The smallest positive root $\wt{r}$, as a function of $q$, defines a curve $p = \wt{r}$ that splits the region $U$ into two sets (note that $\wt{r} = 0$ for $q = 1$). The other two positive roots are always outside the region $U$.
	In conclusion $\overline{K}(\arccos(\zeta^{\pm})) > 0$ for $(p,q) \in U$ only if $p > \wt{r}$, and so iv) is proved.
\end{proof}

\begin{remark}
	When $\delta = 1$ the formula for $K$ coincides with that of \eqref{eq:Kernel:FullIsotropic}-\eqref{eq:Coefficients:FullIsotropic} for fully isotropic elasticity.
\end{remark}

Analogous to \eqref{eq:nonlocalMain1d:Energy}, the energy corresponding to \eqref{eq:nonlocalMain1d:u3Eqn} is
\begin{equation}\label{eq:nonlocalMain1d:Energy:u3Eqn}
	E_{\Gamma}(u_3) := \frac{1}{2} \int_{\bbR^2} (\cL u_3) u_3 \, \rmd x + \int_{\bbR^2} W(u_3) \, \rmd x\,,
\end{equation}
with $E_{\Gamma}(u_3;B_R)$ defining the localized version. The goal now is to show existence, uniqueness, and properties of solutions to \eqref{eq:nonlocalMain1d:u3Eqn}. Thanks to the properties of $K$, solutions to \eqref{eq:nonlocalMain1d:u3Eqn} have a 1D profile, which we prove in \Cref{sec:1DProfile}.

\section{Slip plane parallel to the plane of isotropy}\label{sec:parallel}

In this section, we study case III; $\Gamma$ is taken parallel to the plane of isotropy, that is, $\Gamma = \{ (x_1,x_2,x_3) \, : \, x_3 = 0 \}$. We first introduce the 3D model and reduce to the 2D nonlocal system, and then in \Cref{subsec:para:2DRed} we derive the scalar reduced equation in this setting. We summarize our third result on the existence and uniqueness of the solution for case III in Proposition \ref{prop:parallel}. We characterize the explicit form and the positivity of the corresponding nonlocal kernel in Section  \ref{subsec:para:2DRed}.

The Dirichlet-to-Neumann map and the subsequent 1D reduced equation turn out to have forms identical to their analogues in the setting of full isotropic elasticity. The analysis of \cite{dong2021existence} then readily applies, with one exception;
the analogue of the Poisson's ratio $\nu$
here is instead a function of the five elastic constants. 
The constraint on the Poisson's ratio in the isotropic case from \cite{dong2021existence} here describes a 
region in $\bbR^5$ for which the integro-differential kernel is positive.
The notation for the full 3D system and the perturbed and localized energies is retained.

\begin{lemma}
	Assume that $\bu \in C^2(\bbR^3 \setminus \Gamma ; \bbR^3)$, where $\Gamma = \{ x_3 = 0 \}$ satisfies 
	\begin{equation}\label{eq:BCs:ParallelCase}
		\begin{split}
			u_1^+(x_1,x_2) &= - u_1^-(x_1,x_2)\,, \\
			u_2^+(x_1,x_2) &= -u_2^-(x_1,x_2)\,, \\
			u_3^+(x_1,x_2) &= u_3^-(x_1,x_2)\,, \\
		\end{split}
	\end{equation}
	and satisfies either
	\begin{equation}\label{eq:FarFieldCond:ParallelCase:u1}
		u_1^+(\pm \infty,x_2) = \pm 1 \quad \text{ for any } x_2 \in \bbR
	\end{equation}
	or
	\begin{equation}\label{eq:FarFieldCond:ParallelCase:u2}
		u_2^+(x_1,\pm \infty) = \pm 1 \quad \text{ for any } x_1 \in \bbR\,.
	\end{equation}
	Suppose additionally that $\bu$ is a local minimizer of the total energy
	\begin{equation*}
		E(\bu) = E_{\mathrm{els}}(\bu) + \int_{\Gamma} W(u_1^+,u_2^+) \, \rmd S
	\end{equation*}
	in the sense of \Cref{def:perp:LocalMin}. Then $\bu$ satisfies the Euler-Lagrange equations
	\begin{equation}\label{eq:ELEqn:ParallelCase}
		\begin{gathered}
			\bbL \bu = 0 \text{ in } \bbR^3 \setminus \Gamma\,, \\
			\sigma_{13}^+ + \sigma_{13}^- = \p_{u_1} W(u_1^+,u_2^+) \text{ on } \Gamma\,, \\
			\sigma_{23}^+ + \sigma_{23}^- = \p_{u_2} W(u_1^+,u_2^+) \text{ on } \Gamma\,, \\
			\sigma_{33}^+ - \sigma_{33}^- = 0 \text{ on } \Gamma\,.
		\end{gathered}
	\end{equation}
\end{lemma}

\begin{proposition}[Case III: $\Gamma$ parallel to the plane of isotropy, $\p_{u_1}W = 0$]\label{prop:parallel}
	Suppose that the five elastic constants satisfy \eqref{eq:UniformEllipticityConditions}, and suppose further that $(C_{11},C_{33},C_{13},C_{44},C_{66}) \in \cV^+$, where $\cV^+ \subset \bbR^5$ is the region
	\begin{equation}\label{eq:para:CoeffCond2}
		\cV^+ := \left\{ (C_{11},C_{33},C_{13},C_{44},C_{66}) \, : \, \frac{2}{3} < \frac{\eta_1}{\eta_2} < \frac{3}{2} \right\}\,,
	\end{equation}
	where $\eta_1$ and $\eta_2$ are the real numbers defined as
	\begin{equation}\label{def:Kappa1}
		\eta_1 := 2 \sqrt{C_{44} C_{66}}
	\end{equation}
	and
	\begin{equation}\label{def:Kappa2}
		\eta_2 := \frac{1}{\tau} \Bigg( C_{11} - C_{13} - C_{44} + \frac{C_{13} C_{44}}{C_{33}} + \frac{ \sqrt{C_{11} C_{33}} (C_{33} - C_{13}) (C_{44} + C_{13}) }{C_{33}^2} \Bigg)\,,
	\end{equation}
	and where
	\begin{equation}\label{def:Kappa3}
		\tau := \frac{\sqrt{\sqrt{C_{11} C_{33} } - C_{13}} \sqrt{ \sqrt{C_{11} C_{33}} + C_{13} + 2 C_{44} } }{2 \sqrt{C_{33} C_{44}} }\,.
	\end{equation}
	Assume that $\p_{u_1}W = 0$, i.e. $W(u_1,u_2) = W(u_2)$, and that $W$ satisfies \eqref{potential}. Then the system \eqref{eq:ELEqn:ParallelCase} has a unique 
	(up to translation in the $x_1$-direction and rotation about the $x_3$-axis) 
	classical solution $\bu : \bbR^3 \to \bbR^3$ belonging to $C^{\infty}(\bbR^3 \setminus \Gamma;\bbR^3)$ with $\bu(\cdot,\cdot,x_3) \in H^s(\bbR^2)$ for all $s \geq 1$ and for all fixed $x_3 \neq 0$. Moreover, the solution is a local minimizer of $E$ in the sense of \Cref{def:perp:LocalMin}. The second component $u_2$ is the unique stable solution to \eqref{eq:Nonlocal1DEqn:Parallel} with specific nonlocal kernel defined in \eqref{eq:Iso1:KernelDefn1} and $u_2$ has a 1D profile.
\end{proposition}

Now we reduce the 3D system to a 2D system just as in  \Cref{sec:3DModel:perp}. Here $x = (x_1,x_2) \in \bbR^2$ and $k = (k_1,k_2) \in \bbR^2$.

\begin{lemma}\label{lma:FindingGeneral2DSystem:ParallelCase}
	Suppose that the five elastic constants satisfy \eqref{eq:UniformEllipticityConditions}.
	Suppose that
	$\bu$ satisfies \eqref{eq:ELEqn:ParallelCase} and remains bounded as $|x_3| \to \infty$, and assume that $u_1^{+}$ and $u_2^{+}$ belong to $H^s(\bbR^2)$ for some $s \geq 1/2$. Then $\bu$ can be expressed entirely in terms of $u_1^+$ and $u_2^+$. In particular, $(u_1^+, u_2^+)$ satisfies the nonlocal system
	\begin{equation}\label{eq:2DEqn:Formal:ParallelCase}
		\begin{bmatrix}
			\cF (\sigma_{13}^+(k)+\sigma_{13}^-(k))  \\
			\cF (\sigma_{23}^+(k)+\sigma_{23}^-(k))  \\
		\end{bmatrix}
		:= -\bbA(k) 
		\begin{bmatrix}
			\hat{u}_1^+(k) \\
			\hat{u}_2^+(k) \\
		\end{bmatrix}
		= \cF 
		\begin{bmatrix}
			\p_{u_1} W (u_1^+,u_2^+) \\
			\p_{u_2} W (u_1^+,u_2^+) \\
		\end{bmatrix} \quad \text{ on } \Gamma\,,
	\end{equation}
	where the $2 \times 2$ matrix $\bbA(k)$ is given by
	\begin{equation}\label{eq:2DSystem:ParallelCase}
		\bbA(k) = \eta_1 |k| \bbI + (\eta_2-\eta_1) |k| \frac{k \otimes k}{|k|^2}\,,
	\end{equation}
	and where $\bbI$ denotes the $2 \times 2$ identity matrix and $(k \otimes k)_{ij} = k_i k_j$ for $i,j=1,2$.
\end{lemma}

\begin{proof}
The proof is a straightforward computation using the process outlined in \Cref{lma:FindingGeneral2DSystem}.
The $2 \times 2$ matrix $\bbA$ is determined by the $6 \times 6$ matrix $\frak{A}$.
The characteristic polynomial of $\frak{A}(k)$ is given by
\begin{equation*}
	\wt{p}(\lambda) := \left( \lambda^2 - \delta |k|^2 \right) \left( \alpha \lambda^4 - \gamma |k|^2 \lambda^2 + \beta |k|^4 \right)\,,
\end{equation*}
where the four dimensionless parameters $\alpha$, $\beta$, $\gamma$ and $\delta$ are given by
\begin{equation*}
	\alpha = \frac{C_{33}}{C_{44}}\,, \quad \beta = \frac{C_{11}}{C_{44}}\,, \quad \gamma = 1 + \alpha \beta - \left( \frac{C_{13}}{C_{44}}+1 \right)^2\,, \quad \delta = \frac{C_{66}}{C_{44}}\,.
\end{equation*}
From the uniform ellipticity conditions \eqref{eq:UniformEllipticityConditions}, we see that the constants satisfy at least $\alpha > 0$, $\beta > 0$, $\delta > 0$ and $ - 2 \sqrt{\alpha \beta} < \gamma \leq \alpha \beta + 1$. See also \cite{payton2012elastic} for a detailed discussion on the ellipticity properties of $\bbL$. 
The six roots of $\wt{p}(\lambda)$ are
\begin{equation*}
	\pm r_i(k) = \pm \vartheta_i |k|\,, \qquad i=1,2,3\,,
\end{equation*}
where the (possibly complex) constants $\vartheta_i$ are given by
\begin{equation*}
	\pm \vartheta_1 = \pm \sqrt{\delta}\,, \qquad \pm \vartheta_2 = \pm \sqrt{\frac{\gamma + \sqrt{\gamma^2 - 4 \alpha \beta} }{2 \alpha}}\,, \qquad \pm \vartheta_3 = \pm \sqrt{\frac{\gamma - \sqrt{\gamma^2 - 4 \alpha \beta} }{2 \alpha}}\,.
\end{equation*}
In terms of the original elastic constants, $\vartheta_2 = \tau + \tilde{\tau}$ and $\vartheta_3 = \tau - \tilde{\tau}$, where $\tau$ is defined in \eqref{def:Kappa3} and $\tilde{\tau}$ is defined as
\begin{equation*}
	\tilde{\tau} := \frac{\sqrt{\sqrt{C_{11} C_{33} } + C_{13}} \sqrt{ \sqrt{C_{11} C_{33}} - C_{13} - 2 C_{44} } }{2 \sqrt{C_{33} C_{44}} }\,.
\end{equation*}

The exact form of $\frak{A}$ will vary depending on whether the roots repeat or are complex, but the formula \eqref{eq:2DSystem:ParallelCase} for the reduced matrix $\bbA$ holds for all cases.
\end{proof}

The matrix $\bbA(k)$ is positive definite for all $k \neq 0$ so long as
\begin{equation*}
	\eta_1 > 0 \text{ and } \eta_2 > 0\,.
\end{equation*}
Clearly $\eta_1 > 0$ by the ellipticity conditions on $C_{44}$ and $C_{66}$, but the region in $\bbR^4$ describing the range of $C_{11}$, $C_{13}$, $C_{33}$ and $C_{44}$ for which $\eta_2 > 0$ is more difficult to describe. However, it is a nonempty set; in fact it can be checked that $\eta_2$ is positive for the materials whose elastic constants have been determined in \cite[pg. 3]{payton2012elastic}.
A special case is when $\sqrt{C_{11} C_{33}} - C_{13} - 2 C_{44} = 0$; it is straightforward to check that $\eta_2 > 0$ exactly when
\begin{equation*}
	\begin{split}
		\frac{1}{\alpha} < \beta < 6 + \frac{2}{\alpha} + \alpha + \frac{2 (1+\alpha)}{\alpha} \sqrt{ 1+4 \alpha } &\text{ when } 0 < \alpha \leq 2 + \sqrt{5}\,, \\
		6 + \frac{2}{\alpha} + \alpha - \frac{2 (1+\alpha)}{\alpha} \sqrt{ 1+4 \alpha } < \beta < 6 + \frac{2}{\alpha} + \alpha + \frac{2 (1+\alpha)}{\alpha} \sqrt{ 1+4 \alpha } &\text{ when } 2 + \sqrt{5}  < \alpha\,.
	\end{split}
\end{equation*}

\subsection{The 1D Equation}\label{subsec:para:2DRed}

Clearly the reduced 2D operator is isotropic, and so it will suffice to consider the case
\begin{equation*}
	W(u_1^+,u_2^+) = W(u_2^+)\,;
\end{equation*}
the other case
\begin{equation*}
	W(u_1^+,u_2^+) = W(u_1^+)
\end{equation*}
is equivalent to the first case via a rotation of coordinates.
Similar to the previous sections, we drop the superscripts on $u_1^+$, $u_2^+$ for simplicity. Our reduced nonlocal system reads
\begin{equation}\label{eq:2DEqn:ParallelCase:Case1}
	\left(\eta_1 |k| \bbI + (\eta_2-\eta_1) |k| \frac{k \otimes k}{|k|^2}
	\right) 
	\begin{bmatrix}
		\hat{u}_1(k) \\
		\hat{u}_2(k) \\
	\end{bmatrix}
	= -\cF 
	\begin{bmatrix}
		\p_{u_1} W (u_1,u_2) \\
		\p_{u_2} W (u_1,u_2) \\
	\end{bmatrix} \quad \text{ on } \Gamma\,.
\end{equation}
If we assume additionally that $\p_{u_1} W = 0$, i.e. $W(u_1,u_2) = W(u_2)$, then we can solve for $\hat{u}_1$ in terms of $\hat{u}_2$ and substitute into the remaining equation. We obtain the one-dimensional scalar anisotropic nonlocal equation in Fourier space
\begin{equation}\label{eq:OneDimEqnForU2:ParallelCase}
	\cF(\cL u_2)(k_1,k_2) := \frac{\det \bbA(k)}{a_{11}(k)} \hat{u}_2(k) = - \cF (W'(u_2))\,.
\end{equation}
Inverting the Fourier transform and using \eqref{eq:FarFieldCond:ParallelCase:u1},
we obtain the scalar nonlocal equation
\begin{equation}\label{eq:Nonlocal1DEqn:Parallel}
	\begin{gathered}
		\cL u_2 (x_1,x_2) = - W'(u_2(x_1,x_2))\,, \quad (x_1,x_2) \in \Gamma\,, \\
		\lim\limits_{x_2 \to \pm \infty} u_2(x_1,x_2) = \pm 1\,, \quad x_1 \in \bbR\,,
	\end{gathered}
\end{equation}
where
\begin{equation}\label{eq:OperatorSymbol:parallel}
		\cF(\cL u_2)(k) = \frac{\eta_1 \eta_2 |k|^3}{\eta_2 k_1^2 + \eta_1 k_2^2 }  \hat{u}_2(k)\,.
\end{equation}
If $\eta_i >0$ for $i=1$, $2$, this equation has the same form as the 1D reduced equation in the case of isotropic elasticity \cite{dong2021existence}.
Setting $\mu = \eta_1/2$, $p = 1$ and $q = \eta_1/\eta_2$ with $k_2$ in place of $k_3$ in \Cref{prop:AlmostIso:1DOperator:IntegralForm}, we obtain the following:

\begin{proposition}\label{prop:IntegralForm:Parallel}
	The integral form of $\cL$ defined in \eqref{eq:OneDimEqnForU2:ParallelCase} is
	\begin{equation*}
		\cL w(x) = - \frac{1}{4 \pi} \int_{\bbR^2} \big( w(x-y)-w(x+y) - 2 w(x) \big) K(y) \, \rmd y\,,
	\end{equation*}
	where 
	\begin{equation}\label{eq:Iso1:KernelDefn1}
		K(z) = \eta_1 \eta_2 \frac{A z_1^4 + Bz_1^2 z_2^2 + C z_2^4 }{|z| ( \eta_1 z_1^2 + \eta_2 z_2^2)^{3}}
	\end{equation}
	with
	\begin{equation}\label{eq:Iso1:KernelDefn2}
		\begin{gathered}
			A= 3\eta_1^2 - 2\eta_1 \eta_2, \quad B= 2 ( 3 \eta_1^2 - 5 \eta_1 \eta_2 + 3 \eta_2^2), \quad C= 3\eta_2^2 -2 \eta_1 \eta_2
		\end{gathered}
	\end{equation}
	is the unique solution in $C^{\infty}(\bbR^2 \setminus \{0\})$ to the PDE
	\begin{equation}\label{eq:Iso1:KernelEqn}
		(\eta_2  \p_1 +  \eta_1 \p_2) K(z) = \eta_1 \eta_2 ( \p_1^2 + \p_2^2) \left[ \frac{1}{|z|^3} \right]\,.
	\end{equation}
\end{proposition}

\begin{proposition}[Homogeneity and positivity of kernel for case III]\label{lma:Iso1:KernelProperties}
	Define $\eta_1$ and $\eta_2$ as in \eqref{def:Kappa1}-\eqref{def:Kappa2}.
	Let the kernel $K : \bbR^2 \setminus \{0\} \to \bbR$ be defined as in \eqref{eq:Iso1:KernelDefn1}-\eqref{eq:Iso1:KernelDefn2}. 
	Then $K$ satisfies the following:
	\begin{enumerate}
		\item[i)] $K(x) = K(-x)$, $K(\rho x) = \rho^{-3} K(x)$ for $\rho > 0$.
		\item[ii)] $|x^{a} D^a K(x)| < \frac{C}{|x|^3}$ for any multi-index $a \in \bbN_0^2$ and any $x \neq 0$.
		\item[iii)] $K$ is the unique solution to the equation \eqref{eq:Iso1:KernelEqn} in $C^{\infty}(\bbR^2 \setminus \{0\})$ that is homogeneous of order $-3$.
		\item[iv)] Suppose that the five elastic constants are in $\cV^+$. 
		Then there exist positive constants $c$ and $C$ depending only on the five elastic constants such that $\frac{c}{|x|^3} \leq K(x) \leq \frac{C}{|x|^3}$.
	\end{enumerate}
\end{proposition}

\begin{proof}
	We proceed identically to \Cref{lma:AlmostIso:KernelProperties}. Set $\ell = \eta_1 / \eta_2$.
	The roots of $\overline{K}'(\theta) = \frac{d}{d\theta} K(\cos(\theta),\sin(\theta))$ for $\theta \in [0,\pi/2]$ are  
	are $0$, $\pi/2$, and $\arccos(\zeta)$, where
	\begin{equation*}
		\zeta := \sqrt{\frac{1 +2 \ell - 2 \sqrt{1+\ell^2}}{\ell-1}} \quad \text{ for } \ell \in (1/2,3/4) \cup (4/3,2)\,.
	\end{equation*}
	In the range $\frac{1}{2} < \ell < 2$, $\zeta \in [0,1]$ only for $\ell \in (1/2,3/4) \cup (4/3,2)$.
	Therefore, the equation $\overline{K}'(\theta) = 0$ has exactly three solutions in $[0,\pi/2]$ for $\ell \in (1/2,3/4) \cup (4/3,2)$, and exactly two solutions otherwise. 
	Therefore $K(x_1,x_2) >0 $ if and only if $\bar{K}(\theta)>0$ if and only if
	\begin{equation*}
		\overline{K}(0) >0\,, \quad \overline{K}(\pi/2) > 0\,, \quad \text{ and } \quad \overline{K}(\arccos(\zeta)) > 0 \text{ when } \ell \in (1/2,3/4) \cup (4/3,2)\,.
	\end{equation*}
	Substituting directly gives the relations
	\begin{equation}\label{eq:Iso1:Pf1}
		\begin{split}
			\frac{2 (3 \ell- 2)}{\ell^2} &>0\,, \\
			2(3-2\ell) &>0\,, \\
			\frac{\ell^3+\ell^2 \sqrt{\ell^2+1} +\sqrt{\ell^2+1}+1}{2 \ell^2} &> 0\text{ when } \frac{4}{3} < \ell < 2 \text{ or } \frac{1}{2} < \ell < \frac{3}{4}\,.
		\end{split}
	\end{equation}
	Clearly the third relation in \eqref{eq:Iso1:Pf1} is always satisfied, and the other two relations give the condition $2/3 < \ell < 3/2$.
\end{proof}

Note that taking the elastic constants as in \eqref{eq:ConsistentWithIsotropy} gives $\frac{\eta_1}{\eta_2} = 1- \nu$.

\section{Bounded stable solutions have 1D profiles}\label{sec:1DProfile}

	We now prove well-posedness and properties of solutions to the reduced scalar equations. We simultaneously consider the following three settings:
	\begin{enumerate}
		\item[I)] $W(u_1,u_3) = W(u_1)$, $u = u_1$ satisfies \eqref{eq:nonlocalMain1d}, the kernel $K$ is defined in \Cref{prop:KernelForm:perp:u1Dep}, and the elastic constants satisfy the assumptions of \Cref{prop:Perp:u1Dependence}.
		
		\item[II)] $W(u_1,u_3) = W(u_3)$, $u = u_3$ satisfies \eqref{eq:nonlocalMain1d:u3Eqn}, the kernel $K$ is defined in \Cref{lma:AlmostIso:KernelProperties}, and the elastic constants satisfy the assumptions of \Cref{prop:Perp:u3Dependence}.
		
		\item[III)] $W(u_1,u_2) = W(u_2)$, $u = u_2$ satisfies \eqref{eq:Nonlocal1DEqn:Parallel}, the kernel $K$ is defined in \Cref{prop:IntegralForm:Parallel}, and the elastic constants satisfy the assumptions of \Cref{prop:parallel}.
	\end{enumerate}
	Thanks to the work in the previous sections, each of these settings can be cast in the mold of the following problem: find $u$ satisfying 
	\begin{equation}\label{eq:Nonlocal1DEqn:GeneralCase}
		\begin{gathered}
			\cL u (x_1,x_2) = - W'(u(x_1,x_2))\,, \quad (x_1,x_2) \in \bbR^2\,, \\
			\lim\limits_{x_1 \to \pm \infty} u(x_1,x_2) = \pm 1\,, \quad x_2 \in \bbR\,,
		\end{gathered}
	\end{equation}
	where $x = (x_1,x_2)$,
	\begin{equation}\label{eq:Nonlocal1DOperator:GeneralCase}
		\cL w(x) = - \frac{1}{4 \pi} \int_{\bbR^2} \big( w(x-y)-w(x+y) - 2 w(x) \big) K(y) \, \rmd y\,,
	\end{equation}
	and where the nonlocal kernel $K \in C^{\infty}(\bbR^2 \setminus \{0\})$ satisfies
	\begin{enumerate}[(i)]
		\item $K(x) = K(-x)$, $K(\rho x) = \rho^{-3} K(x)$ for any $\rho > 0$,
		\item $|x^{a} D^a K(x)| < \frac{C}{|x|^3}$ for any multi-index $a \in \bbN_0^2$ and any $x \neq 0$,
		\item there exist positive constants $\ell$ and $L$ such that $\frac{\ell}{|x|^3} \leq K(x) \leq \frac{L}{|x|^3}$ for all $x \in \bbR^2 \setminus \{0\}$.
	\end{enumerate}
	
	In this section, we show that any bounded stable solution in each setting I), II) and III) has a 1D profile, i.e., $u(x)=\psi(e\cdot x)$ for some $e\in \bbS^1$, where $\psi$ is the unique (up to translations) solution to a 1D scalar problem.

	Let us first clarify the definition of stable solutions.
	\begin{definition}\label{def_stable}
		We say that $u$ is a stable solution to \eqref{eq:Nonlocal1DEqn:GeneralCase} if  
		\begin{equation*}
			\int_{\bbR^2} \left({\cL} v +  W''(u)v \right) v  \, \rmd x \geq 0 \quad \text{ for any } v\in C_c^2(\bbR^2).
		\end{equation*}
	\end{definition}
	
	Define the total energy of $u$ in any Euclidean ball $B_{R}\subset \bbR^2$ as
	\begin{equation}\label{B_energy}
		\begin{aligned}
			E_\Gamma^0(u; B_R)
			&:=\frac{C_d}{4}\iint_{\bbR^2\times \bbR^2\backslash B_R^c\times B_R^c} |u(x)-u(y)|^2 {K}(x-y) \, \rmd x \, \rmd y +  \int_{B_R} W(u(x)) \, \rmd x\\
			&:=\frac{C_d}{4}\cE(u;B_R)+ F(u;B_R),
		\end{aligned}
	\end{equation}
	where ${K}$ satisfies properties (i)-(ii)-(iii) above.
		In this general setting we can obtain the interior BV estimate for stable solutions. 
	\begin{lemma}\label{BV}
		Let $|u|\leq M$ be a bounded stable solution to \eqref{eq:Nonlocal1DEqn:GeneralCase} in the sense of  \Cref{def_stable}.
		Assume {that} $W$ satisfies \eqref{potential} and  $L_*:=\max\{2, \|W\|_{C^{2,a}_b(\bbR)}\}$. Then there exists a constant $C = C(\ell,L, M, L_*)$ such that for any $B_R\subset \bbR^2$ and $R\geq 1$,
		\begin{equation}\label{uniform-e}
			\int_{B_R} |\nabla u| \, \rmd x \leq C(\ell, L, M, L_*) R \log (L_*R ), \qquad \cE(u, B_R)\leq C(\ell, L, M, L_*) R \log^2(L_*R).
		\end{equation}
	\end{lemma}
	We omit the proof of this lemma because it follows the same procedures as \cite[Proposition 4.5]{dong2021existence}. The main idea is to first obtain the key interior BV estimate for any direction $e \in \bbS^1$
	\begin{equation}\label{BVii}
		\begin{gathered}
		\left( \int_{B_{\frac12}}(\partial_{e} u(x))_+  \, \rmd x  \right)  \left( \int_{B_{\frac12}}(\partial_{e} u(y))_-  \, \rmd y \right)
		\leq C(\delta, \nu) \frac{\cE(u, B_R)}{R^2}, \,\, \\
		\int_{B_{\frac12}} |\nabla u(x)|  \, \rmd x 
		\leq C(\delta, \nu) (1+\sqrt{\cE(u,B_1)}), 
		\end{gathered}
	\end{equation}
	which relies on properties (i)-(ii)-(iii) of the kernel $K$, and then to use a sharp interpolation inequality for energy $\cE(u, B_R)$.
	Combining the energy estimate \eqref{uniform-e} with the interior BV estimate \eqref{BVii}, we obtain that as $R\to +\infty$, for any direction $e\in \bbS^1$ and any half ball in $\bbR^2$ we have
	\begin{equation*}
		\partial_e u \geq 0  \text{ in }\bbR^2 \quad \text{ or } \quad \partial_{e} u \leq 0 \text{ in } \bbR^2 \qquad \text{ for any }e\in \bbS^1,
	\end{equation*}
	which yields the conclusion {that} $u$ {has} a 1D monotone profile. 
	
	Thus for each setting I), II), and III), we conclude that
	\begin{theorem}\label{thm6.3}
		Assume that $|u|\leq M$ is a bounded stable solution to \eqref{eq:Nonlocal1DEqn:GeneralCase} and $W$ satisfies \eqref{potential}. Suppose the kernel $K$ satisfies conditions (i)-(ii)-(iii) above.
		Then $u$ has a 1D monotone profile and $|u| \leq 1$.
		Moreover, the solution to \eqref{eq:Nonlocal1DEqn:GeneralCase} can be characterized as $u(x) = \psi(e\cdot x)$ for any $e := (\cos\theta,\sin\theta) \in \bbS^{1}$ with fixed $\theta \in (-\frac{\pi}{2},\frac{\pi}{2})$, where $\psi$ is the unique solution (up to translations) to the 1D scalar problem
		\begin{equation*}
			\begin{split}
				(-\Delta)^{\frac{1}{2}} \psi(x_1) &= \frac{-1}{m(\cos\theta,\sin\theta)} W'(\psi(x_1))\,, \quad x_1 \in \bbR\,, \\
				\lim\limits_{x_1 \to \pm \infty} \psi(x_1) &= \pm 1\,, 
			\end{split}
		\end{equation*}
		and where $m$ is the symbol associated to $\cL$ that is defined in \eqref{eq:1DFourierSymbolCase1}, \eqref{eq:OperatorSymbol:perp:u3Dep}, or \eqref{eq:OperatorSymbol:parallel}.
	\end{theorem}

	\begin{proof}
		The proof is identical to that of \cite[Theorem 4.6]{dong2021existence}.
	\end{proof}

\begin{proof}[proofs of \Cref{prop:Perp:u1Dependence}, \Cref{prop:Perp:u3Dependence}, and \Cref{prop:parallel}]
		 From Theorem \ref{thm6.3}, the solution to \eqref{eq:Nonlocal1DEqn:GeneralCase} is unique and has a 1D profile. That means  the solutions to \eqref{eq:nonlocalMain1d}, \eqref{eq:nonlocalMain1d:u3Eqn} and \eqref{eq:Nonlocal1DEqn:Parallel} is unique and has a 1D profile respectively. 
	From obtained one component of the 3D solution, we can further solve for the other two components in the 3D system by using the elastic extension based on the relevant Dirichlet-to-Neumann map in Section 2 (see \cite{gao2021revisit} for the details in the fully isotropic case). Finally, the unique stable solution to the full 3D system corresponding to each of the three settings is obtained.
\end{proof}

\begin{remark}
	For more general anisotropic materials, or for elastic coefficients not satisfying \eqref{eq:Perp:Coeff:Cond2}, the expression of the matrix $\bbA$ associated to the Dirichlet-to-Neumann map is more complicated, making the program followed in this work less tractable. The results in this paper remain open in those more general settings.
\end{remark}

\bibliography{References}

\begin{thebibliography}{10}

\bibitem{CS05}
Xavier Cabré and Joan Solà-Morales.
\newblock Layer solutions in a half-space for boundary reactions.
\newblock {\em Communications on Pure and Applied Mathematics},
  58(12):1678–1732, Dec 2005.

\bibitem{CSV19}
Eleonora Cinti, Joaquim Serra, and Enrico Valdinoci.
\newblock Quantitative flatness results and {$BV$}-estimates for stable
  nonlocal minimal surfaces.
\newblock {\em Journal of Differential Geometry}, 112(3):447–504, Jul 2019.

\bibitem{dipierro2020improvement}
Serena Dipierro, Joaquim Serra, and Enrico Valdinoci.
\newblock Improvement of flatness for nonlocal phase transitions.
\newblock {\em American Journal of Mathematics}, 142(4):1083--1160, 2020.

\bibitem{dong2021existence}
Hongjie Dong and Yuan Gao.
\newblock Existence and uniqueness of bounded stable solutions to the
  peierls--nabarro model for curved dislocations.
\newblock {\em Calculus of Variations and Partial Differential Equations},
  60(2):1--26, 2021.

\bibitem{FS20}
Alessio Figalli and Joaquim Serra.
\newblock On stable solutions for boundary reactions: a {De Giorgi}-type result
  in dimension 4 + 1.
\newblock {\em Inventiones mathematicae}, 219(1):153–177, Jan 2020.

\bibitem{gao2021existence}
Yuan Gao, Jian-Guo Liu, and Zibu Liu.
\newblock Existence and rigidity of the vectorial {P}eierls--{N}abarro model
  for dislocations in high dimensions.
\newblock {\em Nonlinearity}, 34(11):7778, 2021.

\bibitem{gao2021revisit}
Yuan Gao, Jian-Guo Liu, Tao Luo, and Yang Xiang.
\newblock Revisit of the {P}eierls-{N}abarro model for edge dislocations in
  {H}ilbert space.
\newblock {\em Discrete \& Continuous Dynamical Systems-B}, 26(6):3177, 2021.

\bibitem{gao2022asymptotic}
Yuan Gao and Jean-Michel Roquejoffre.
\newblock Asymptotic stability for diffusion with dynamic boundary reaction
  from {G}inzburg-{L}andau energy.
\newblock {\em arXiv preprint arXiv:2201.02105}, 2022.

\bibitem{Mon1}
M.~Gonz\'alez and R.~Monneau.
\newblock Slow motion of particle systems as a limit of a reaction-diffusion
  equation with half-{L}aplacian in dimension one.
\newblock {\em Discrete Contin. Dyn. Syst.}, 32:1255--1286, 2012.

\bibitem{Gui19}
Changfeng Gui and Qinfeng Li.
\newblock Some energy estimates for stable solutions to fractional
  {Allen--Cahn} equations.
\newblock {\em Calculus of Variations and Partial Differential Equations},
  59(2):49, 2020.

\bibitem{HL}
J.~P. Hirth and J.~Lothe.
\newblock {\em Theory of Dislocations (2nd ed.)}.
\newblock Wiley, New York, 1982.

\bibitem{merodio2003note}
J~Merodio and RW~Ogden.
\newblock A note on strong ellipticity for transversely isotropic linearly
  elastic solids.
\newblock {\em Quarterly Journal of Mechanics \& Applied Mathematics}, 56(4),
  2003.

\bibitem{nabarro1947dislocations}
F.R.N. Nabarro.
\newblock Dislocations in a simple cubic lattice.
\newblock {\em Proceedings of the Physical Society}, 59(2):256--272, 1947.

\bibitem{PV1}
S.~Patrizi and E.~Valdinoci.
\newblock Relaxation times for atom dislocations in crystals.
\newblock {\em Calc. Var. Partial Differ. Equ.}, 55:1--44, 2016.

\bibitem{PV2}
S.~Patrizi and E.~Valdinoci.
\newblock Long-time behavior for crystal dislocation dynamics.
\newblock {\em Math. Models Methods Appl. Sci.}, 27:2185--2228, 2017.

\bibitem{payton2012elastic}
RC~Payton.
\newblock {\em Elastic wave propagation in transversely isotropic media},
  volume~4.
\newblock Springer Science \& Business Media, 2012.

\bibitem{peierls1940size}
Rudolf Peierls.
\newblock The size of a dislocation.
\newblock {\em Proceedings of the Physical Society}, 52(1):34, 1940.

\bibitem{savin2018rigidity}
Ovidiu Savin.
\newblock Rigidity of minimizers in nonlocal phase transitions.
\newblock {\em Anal. PDE}, 11(8):1881--1900, 2018.

\bibitem{ting2005poisson}
T.C.T. Ting and Tungyang Chen.
\newblock {P}oisson's ratio for anisotropic elastic materials can have no
  bounds.
\newblock {\em The Quarterly Journal of Mechanics and Applied Mathematics},
  58(1):73--82, 2005.

\bibitem{Xiang_2006}
Yang Xiang.
\newblock Modeling dislocations at different scales.
\newblock {\em Commun. Comput. Phys.}, 1(3):383--424, 2006.

\end{thebibliography}
\bibliographystyle{plain}

\end{document}